\documentclass{amsart}
\usepackage{amsmath, amssymb, amsthm}
\usepackage{fullpage}
\newtheorem{theorem}{Theorem}[section]
\newtheorem{lemma}[theorem]{Lemma}
\newtheorem{corollary}[theorem]{Corollary}

\theoremstyle{definition}
\newtheorem{definition}[theorem]{Definition}
\newtheorem{example}[theorem]{Example}
\theoremstyle{remark}
\newtheorem{remark}[theorem]{Remark}

\usepackage{bm}

\usepackage{soul}
\usepackage[dvipsnames]{xcolor}

\sethlcolor{GreenYellow}

\usepackage{bbm} 
\usepackage{blkarray, bigstrut} 

\usepackage[
backend=biber,
style=numeric,
 sortcites,
 url = false,
 clearlang=true,
 maxnames=99
]{biblatex}

\usepackage{tikz-cd}

\definecolor{WildStrawberry}{RGB}{242, 105, 142}

\newcommand{\htpyeq}{\simeq}

\newcommand{\op}{\textnormal{op}} 

\bibliography{references} 

\usepackage{xcolor}
\colorlet{darkgreen}{green!40!black}
\usepackage{ifpdf}
\ifpdf \usepackage[pdftex]{hyperref}
\else \usepackage[ps2pdf]{hyperref} 
      \usepackage{breakurl}
\fi
\hypersetup{colorlinks=true,urlcolor=blue,citecolor=darkgreen,
linkcolor=darkgreen,linktocpage,bookmarksnumbered,unicode}

\begin{document}

\title{Short, new proofs of Dowker duality} 


\author{Iris H.R. Yoon}             

\email{hyoon@wesleyan.edu}

\address{Department of Mathematics and Computer Science,
         Wesleyan University,
        45 Wyllys Ave, 
         Middletown, CT 06459     
         USA \\
Department of Mathematics and statistics, Swarthmore College, 500 College Ave, PA 19081 USA}

\keywords{Dowker duality, poset fiber lemmas, relational complexes, Nerve Lemma.}

\begin{abstract}
This paper presents three short, new proofs of Dowker duality using various poset fiber lemmas. We introduce modifications of joins and products of simplicial complexes called \emph{relational join} and \emph{relational product} complexes. These relational complexes can be constructed whenever there is a relation between the face posets of simplicial complexes, which includes the context of Dowker duality and covers of simplicial complexes. In this more general setting, we show that the homologies of the simplicial complexes and the relational complexes fit together in a long exact sequence. 
\end{abstract}

\maketitle

\section{Introduction}
Given sets $A$ and $X$ and a relation $R \subseteq A \times X$, we say that $a \in A$ is \emph{related} to $x \in X$ if and only if $(a, x) \in R$. A Dowker complex with base $A$, denoted $D_A$, is an abstract simplicial complex whose simplices are non-empty finite subsets of $A$ that are related to a common element in $X$. One can similarly construct $D_X$, a Dowker complex with base $X$. The geometric realizations of $D_A$ and $D_X$, which are built on two distinct vertex sets, are homotopy equivalent. We refer to this result as \emph{Dowker duality}.

In its original statement, Dowker duality took a weaker form: $H_n(D_A) \cong H_n(D_X)$ for all $n \in \mathbb{Z}_+$  \cite{dowkerHomologyGroupsRelations1952}. 
The stronger version of Dowker duality, i.e., $\|D_A \|\htpyeq \|D_X\|$, was established in \cite{bjornerTopologicalMethods1996} using the Nerve Lemma. Dowker duality has been extended to the functorial setting \cite{chowdhuryFunctorialDowkerTheorem2018} and bifiltrations \cite{hellmer_density_2024}. In recent years, new proofs of Dowker duality have also emerged \cite{robinsonCosheafRepresentationsRelations2022, brunRectangleComplexRelation2022, brun_dowker_2024}. 
Example applications of Dowker complexes include computational neuroscience \cite{vaupel_topological_2023}, cancer biology \cite{yoon_deciphering_2024, stolz_relational_2023}, molecular biology \cite{liu_dowker_2022}, networks \cite{chowdhuryFunctorialDowkerTheorem2018}, PDF parsers \cite{ewing_metric_2021}, and persistence diagrams \cite{yoon_persistent_2023}.

This paper presents three short, new proofs of Dowker duality. The first proof uses Galois connections. For the second and third proofs, we introduce two new complexes called the \emph{relational join complex}, denoted $D_A \star_{\tilde{R}} D_X$, and the \emph{relational product complex}, denoted $D_A \times_{\tilde{R}} D_X$, that are homotopy equivalent to the Dowker complexes (see following diagram, right). The proofs are concise thanks to established poset fiber lemmas applied to the face posets of the relevant complexes (see following diagram, left). 
\[
\begin{tikzcd}[column sep = small]
\quad & \tilde{R} \arrow[dl] \arrow[dr] & \quad \\
P_A \arrow{dr}  \arrow[rr,yshift=-3, swap]  & \, & P_X^{\op} \arrow{dl}  \arrow[ll,  yshift = 3, swap] \\
\, & P_A \star_{R} P_X^{\op} & \, 
\end{tikzcd} 
\hspace{4em}
\begin{tikzcd}[column sep = small]
\quad &   D_A  \times_{\tilde{R}} D_X   \arrow[dl, swap, "\htpyeq"] \arrow{dr}{\htpyeq} & \quad \\ 
\|D_A\|  \arrow{dr}[swap]{\htpyeq}   \arrow[rr, "\htpyeq", yshift=-3, swap]  & \, &   \|D_X\|   \arrow{dl}{\htpyeq}  \arrow[ll, "\htpyeq", yshift = 3, swap] \\
\, &   \|D_A \star_{\tilde{R}} D_X\|   & \, 
\end{tikzcd} 
\]
The construction of the relational complexes is not limited to Dowker complexes. One can build the relational complexes whenever there is a relation $\tilde{R}$ between the face posets of simplicial complexes $K$ and $M$. The paper presents new proofs of the Nerve Lemma using relational complexes. In the general setting of relations between face posets of simplicial complexes, we show that the relational join complex $\|K \star_{\tilde{R}} M\|$ is homeomorphic to a double mapping cylinder involving the relational product complex $K \times_{\tilde{R}}M$ (Theorem~\ref{thm:double_mapping_cylinder}). It then follows that the homologies of $K, M$, the relational join complex $K\star_{\tilde{R}}M$, and the relational product complex $K \times_{\tilde{R}} M$ fit into a long exact sequence (Corollary \ref{cor:LES}) 
\[  \cdots \to H_n(K \times_{\tilde{R}} M) \to H_n(K) \oplus H_n(M) \to H_n(K\star_{\tilde{R}}M) \to H_{n-1}(K \times_{\tilde{R}} M) \to \cdots \; .\]

\medskip

The paper is organized as follows. Section~\ref{sec:preliminaries} contains preliminaries. Section~\ref{sec:short_new_proofs} presents three short new proofs of Dowker duality and introduces the relational complexes. In Section~\ref{section:SpectralSequences}, we establish the double mapping cylinder and the long exact sequence among the relational complexes.

\section{Preliminaries}
\label{sec:preliminaries}

\subsection{Dowker duality}

\begin{definition}
Let $A$ and $X$ be sets. Let $R \subseteq A \times X$ be a relation. The Dowker complex with base $A$, denoted $D_A$, is the abstract simplicial complex whose simplices $\sigma_A$ are non-empty finite collections of elements in $A$ that are related to a common element in $X$. That is, there exists some $x \in X$ such that $\sigma_A \times \{x \} \subseteq R$. 
\end{definition}

One can similarly define $D_X$, the Dowker complex with base $X$. 

Given $a \in A$ and $x \in X$, we indicate that $a$ and $x$ are related by writing $(a, x) \in R$ or $aRx$. When $A$ and $X$ are finite, the relation $R$ is often represented by a binary matrix whose rows and columns correspond to $A$ and $X$. The matrix has entry $1$ if the corresponding elements of $A$ and $X$ are related, and 0 otherwise. 

\begin{theorem}
\label{thm:Dowker}
(Dowker duality \cite{dowkerHomologyGroupsRelations1952, bjornerTopologicalMethods1996})
Let $R \subseteq A \times X$ be a relation between sets. Then $\|D_A\| \htpyeq \|D_X\|$.
\end{theorem}
Here, we use $\|K\|$ to denote the geometric realization of a simplicial complex $K$. Dowker duality has also been generalized to the following functorial version \cite{chowdhuryFunctorialDowkerTheorem2018}. 

\begin{theorem}(Functorial Dowker duality \cite{virk_rips_2021, brunRectangleComplexRelation2022})
Let $A, A', X, X'$ be sets with relations $R \subseteq A \times X$ and $R' \subseteq A' \times X'$.  Let $f: A \to A'$ and $g: X \to X'$ be set maps such that $(f(a), g(x)) \in R'$ for all $(a, x) \in R$. 
Let $D_f: D_A \to D_{A'}$ and $D_g: D_X \to D_{X'}$ be the simplicial maps induced by the maps $f$ and $g$ on the Dowker complexes. There exist homotopy equivalences $\Gamma: \|D_A\| \to \|D_X\|$ and $\Gamma': \|D_{A'}\| \to \|D_{X'}\|$ such that the following diagram commutes up to homotopy.
\[ 
\begin{tikzcd}
\|D_A\| \arrow[r, "\|D_f\|"] \arrow[d, "\Gamma"] & \|D_{A'}\| \arrow[d, "\Gamma'"]\\
\|D_{X}\| \arrow[r, "\|D_g\|"] & \|D_{X'}\|
\end{tikzcd}
\]
\end{theorem}

Note that the above theorem is stronger than the Functorial Dowker Theorem from {\cite{chowdhuryFunctorialDowkerTheorem2018}}. There, $R$ and $R'$ must be nested relations of a fixed $A' \times X'$.

Dowker duality is closely related to a fundamental result in topology called the Nerve Lemma. Given a simplicial complex $K$, let $\mathcal{U}$ denote a cover of $K$ by subcomplexes. Let $N\mathcal{U}$ denote the nerve of the cover. A cover $\mathcal{U}$ is called a \emph{good cover} if every non-empty finite intersection of its elements is contractible. 

\begin{lemma}(Nerve Lemma)
Let $K$ be a simplicial complex, and let $\mathcal{U}$ be a good cover of $K$ by subcomplexes. Then, $K$ and $N\mathcal{U}$ are homotopy equivalent.   
\end{lemma}

See \cite{bauer_unified_2023} for details and \cite{bjornerTopologicalMethods1996} for a proof of Dowker duality using the Nerve Lemma.

\subsection{Posets, Galois connections, and fiber lemmas}
The order complex of a poset allows us to study the poset via abstract simplicial complexes.

\begin{definition}
Given a poset $P$, the order complex $\Delta P$ is the abstract simplicial complex whose $k$-simplices are the $k$-chains $p_0 < p_1 < \cdots < p_k$ in $P$. 
\end{definition}

The geometric realization $\Vert \Delta P \Vert$ associates a topological space to a poset $P$. Furthermore, if $P^{\op}$ is the dual poset (obtained from $P$ by reversing all orders), then $\Delta P$ and $\Delta P^{\op}$ are identical abstract simplicial complexes. 

Given a simplicial complex $K$, we use $P_K$ to denote its face poset, i.e., the set of simplices of $K$ ordered by inclusion. Note that $\Delta P_K$ is the barycentric subdivision of $K$. 

We say that a poset $P$ is contractible if its order complex $\Delta P$ is contractible. If $P$ is a poset and $p \in P$, let 
\begin{align*}
P_{\geq p} &= \{ p' \in P \, | \, p' \geq p \} \\
P_{\leq p} &= \{ p' \in P \, | \, p' \leq p \}.
\end{align*}

All posets of the form $P_{\geq p}$ and $P_{\leq p}$ are contractible since $\Vert \Delta P_{\geq p} \Vert$ and $\Vert \Delta P_{\leq p} \Vert$ are cones with $p$ as their respective apex.

The Carrier Lemma is one of the widely used tools for showing that two maps are homotopic. We first define a carrier. Given simplices $\sigma$ and $\tau$, we use $\sigma \subseteq \tau$ to indicate that $\sigma$ is a face of $\tau$. 

\begin{definition}
\label{def:carrier}
Let $K$ be a simplicial complex and $T$ a topological space.
Let $C$ be an order-preserving assignment of subsets of $T$ to simplices of $K$. That is, for every simplex $\sigma$ of $K$, $C(\sigma) \subseteq T$ such that $C(\sigma) \subseteq C(\tau)$ for all simplices $\sigma \subseteq \tau$ in $K$. A map $f : \Vert K \Vert \to T$ is \emph{carried by} $C$ if $f(\Vert \sigma \Vert) \subseteq C(\sigma)$ for all $\sigma \in K$. 
\end{definition}

\begin{lemma}[Carrier Lemma \cite{bjornerTopologicalMethods1996}] 
\label{lemma:carrierlemma}
If $f, g: \Vert K \Vert \to T$ are carried by $C$ and $C(\sigma)$ is contractible for every simplex $\sigma \in K$, then $f$ and $g$ are homotopic. 
\end{lemma}

\begin{proof}
See \cite{walkerHomotopyTypeEuler1981} for proof. 
\end{proof}

If $f: P \to Q$ is an order-preserving or order-reversing map of posets, one can consider $f$ as a simplicial map $f: \Delta P \to \Delta Q$, and $f$ induces a continuous map $\|f\|: \Vert \Delta P \Vert \to \Vert \Delta Q \Vert$. The following lemma, which one can prove using the Carrier Lemma, will be useful. See \cite{bjornerTopologicalMethods1996} (Theorem 10.11) for proof. 

\begin{lemma}
[Order Homotopy Lemma \cite{bjornerTopologicalMethods1996, quillenHomotopyPropertiesPoset1978}]
\label{lemma:order_htpy_thm}
Let $f, g: P \to Q$ be order-preserving maps (resp. order-reversing maps) such that $f(p) \leq g(p)$ for every $p$. The induced maps $\|f\|, \|g\| : \Vert \Delta P \Vert \to \Vert \Delta Q \Vert$ are homotopic.
\end{lemma}

Throughout this paper, we will use the following lemma to prove the functorial Dowker duality. To avoid notational clutter, we will use $k_0 k_1 \dots  k_l$ to denote the simplex $\{ k_0, k_1, \dots, k_l\}$ of $K$. Note that given a simplicial map $g: K \to K'$ between simplicial complexes, there is an induced simplicial map on the order complexes $\tilde{g}: \Delta P_K \to \Delta P_{K'}$. Explicitly, $\tilde{g}$ maps vertex $ k_0 k_1, \dots, k_l$ of $\Delta P_K$ to $g(k_0 k_1 \dots k_l )$ of $\Delta P_{K'}.$

\begin{lemma}
\label{lemma:functorial_subdivision}

Let $g: K \to K'$ be a simplicial map of simplicial complexes, and let $\tilde{g}: \Delta P_K \to \Delta P_{K'}$ be the induced map on the order complex of face posets. Let $\psi: \| \Delta P_K \| \to \| K \|$ and $\psi': \| \Delta P_{K'} \| \to \|K'\|$ be the usual homeomorphisms between realizations of simplicial complexes and their barycentric subdivisions. Then, the following diagram commutes up to homotopy. 
\[ 
\begin{tikzcd}
\| \Delta P_K \| \arrow[r, "\psi"] \arrow[d, "\| \tilde{g}\|"] & \| K \| \arrow[d, " \| g\|"] \\
\| \Delta P_{K'} \| \arrow[r, "\psi'"] & \| K' \| \\
\end{tikzcd} 
\]
\end{lemma}
Note that the diagram does not necessarily commute because the homeomorphism between the geometric realizations of a simplicial complex and its barycentric subdivision, while canonical, is not natural \cite{Fritsch_Piccinini_1990}. 
\begin{proof}

We use the Carrier Lemma. Given a simplex $\sigma = (k_0 < k_0k_1 < \dots < k_0 k_1 \dots k_l)$ of $\Delta P_K$, let
\[
  C(\sigma) \;=\; \bigl\|\,g(k_0k_1\ldots k_l)\,\bigr\| \;\subseteq\; \|K'\|.
\]
Note $C$ is an order-preserving assignment and $C(\sigma)$ is contractible for all $\sigma \in \Delta P_K$. 

We now show that $C$ carries both $\| g \| \circ \psi$ and $\psi' \circ \| \tilde{g}\|$. Let $\sigma \in \Delta P_K$ be the simplex $k_0 < k_0k_1 < \dots < k_0 k_1 \dots k_l$. For each vertex $k_0 k_1 \dots k_s$ of $\sigma$, the homeomorphism $\psi$ maps $\|k_0 k_1 \dots k_s\|$ of $\|\sigma\|$ to the corresponding barycenter of simplex  $\| k_0 k_1 \dots k_s  \|$ in $\|K\|$ and extends linearly. Every such barycenter lies in $\|k_0k_1\ldots k_l\|$, so
$\psi(\|\sigma\|) \subseteq \|k_0k_1\ldots k_l\|$. Applying $\|g\|$,
\[
  \|g\| \circ \psi(\|\sigma\|)
  \;\subseteq\; \bigl\|\,g(k_0k_1\ldots k_l)\,\bigr\|
  \;=\; C(\sigma).
\]
So $\|g \| \circ \psi$ is carried by $C$.

On the other hand, 
\begin{align*}
\|\tilde{g} \| (\| \sigma \|) &= \| \tilde{g}(k_0 < k_0k_1 < \dots < k_0 \, k_1 \dots k_l ) \| \\
&= \| g(k_0) \leq g(k_0 k_1) \leq \dots \leq g( k_0 \dots k_l ) \|
\end{align*}
is the geometric realization of the simplex in $\Delta P_{K'}$ that have $g(k_0)$, $g(k_0 k_1)$, $\dots,$ $g(k_0 k_1 \dots k_l)$ as the vertex set. Under $\psi'$, each vertex maps to the corresponding barycenter in $\|K'\|$, and all of these barycenters lie in $\|g(k_0k_1\ldots k_l)\|$. Hence, 
\[\psi' \circ \| \tilde{g} \|(\|\sigma\|) \subseteq \| g( k_0 k_1 \dots k_l ) \| = C(\sigma).\] 
So $\psi' \circ \|\tilde{g} \|$ is also carried by $C$. 

By the Carrier Lemma (Lemma~\ref{lemma:carrierlemma}), the diagram commutes up to homotopy. 
\end{proof}

We now introduce Galois connections, which we use for the first proof of Dowker duality. 
\begin{definition} 
An (isotone) Galois connection between two posets $P$ and $Q$ is a pair of order-preserving maps $L: P \to Q$ and $U: Q \to P$ such that 
\[ L(p) \leq q \quad \iff \quad p \leq U(q) \quad \forall p \in P, \, q \in Q.\]
\end{definition}
An equivalent condition is 
\begin{align*}
p  \leq U \circ L (p) & \quad \forall p \in P, \, \text{and} \\
L \circ U (q) \leq q & \quad \forall q \in Q. 
\end{align*}

We denote a Galois connection by $(L, U)$ or $L: P \leftrightarrows Q: U$, and we call $L$ the \emph{lower adjoint} of $U$ and $U$ the \emph{upper adjoint} of $L$. While it is common to denote Galois connections as $(L,R)$, we reserve $R$ to denote relations.

Given a Galois connection between $P$ and $Q$, the order complexes $\Delta P$ and $\Delta Q$ are homotopy equivalent. A weaker version of the lemma recently appeared in \cite{patelObiusHomology2023}, and it is mentioned in \cite{walkerHomotopyTypeEuler1981} that the statement can be proved using the Carrier Lemma. Here is a short proof.

\begin{lemma}
\label{lemma:Galois_htpyeq}
Let $L: P \leftrightarrows Q: U$ be a Galois connection between posets $P$ and $Q$. Then $\Vert \Delta P \Vert \htpyeq \Vert \Delta Q \Vert$.
\end{lemma}

\begin{proof}
Since $(L, U)$ is a Galois connection, $p \leq U \circ L(p)$ for all $p \in P$ and $L \circ U(q) \leq q$ for all $q \in Q$. Let $\mathbbm{1}_P$ and $\mathbbm{1}_Q$ be the identity morphisms on $P$ and $Q$. By the order homotopy lemma (Lemma \ref{lemma:order_htpy_thm}), $\|U\|\circ \|L\|$ is homotopic to $\|\mathbbm{1}_P\|$ and $\|L\| \circ \|U\|$ is homotopic to $\| \mathbbm{1}_Q \|$.
\end{proof}

One can also prove the previous lemma using  Quillen's Poset Fiber Lemma, which will appear throughout the paper.

\begin{lemma} 
[Quillen's Poset Fiber Lemma \cite{quillenHomotopyPropertiesPoset1978}]
\label{lemma:Quillen_fiber} 
Let $f: P \to Q$ be an order-preserving morphism of posets. If $f^{-1}_{\leq q} = \{ p \in P \, | \, f(p) \leq q \}$ is contractible for all $q \in Q$ (resp. $f^{-1}_{\geq q} = \{ p \in P \, | \, f(p) \geq q\}$ is contractible for all $q \in Q$), then $f$ induces a homotopy equivalence $\Vert f \Vert: \Vert \Delta P \Vert \to \Vert \Delta Q \Vert$.
\end{lemma}

Given an order-preserving morphism of posets $f: P \to Q$, we refer to $f^{-1}_{\leq q}$ and $f^{-1}_{\geq q}$ as the \emph{fibers of $f$}. Depending on $f$, sometimes the fibers of the form $f^{-1}_{\leq q}$ will be contractible, and other times, it will be the fibers of the form $f^{-1}_{\geq q}$ that will be contractible (see Lemma~\ref{eq:Galois_fibers_min_max} below).  

\medskip
If $L: P \leftrightarrows Q:U$ is a Galois connection, then the fibers of $L$ and $U$ are contractible specifically because they have maximum and minimum elements.  

\begin{lemma}
\label{eq:Galois_fibers_min_max}
Let $L: P \leftrightarrows Q: U$ be a Galois connection. For any $q \in Q,$ the poset $L^{-1}_{\leq q} = \{p \in P \, | \, L(p) \leq q \}$ has a maximum element, namely $U(q)$. For any $p \in P$, the poset $U^{-1}_{\geq p} = \{q \in Q \, | \, p \leq U(q) \} $ has a minimum element, namely $L(p)$. 
\end{lemma}

In particular, it follows that all fibers $L^{-1}_{\leq q}$ and $U^{-1}_{\geq p}$ are contractible. Lemma \ref{lemma:Galois_htpyeq} then follows from Quillen's Poset Fiber Lemma. 

There is a slight variation of Quillen's Poset Fiber Lemma that is useful when product posets are involved. Given posets $P$ and $Q$, let $P \times Q$ be the product poset with the product order $(p,q) \leq (p', q')$ if $p \leq p'$ and $q \leq q'$. Given a subset $Z$ of a poset $X$, we say that $Z$ is \emph{downward closed} if $x \in Z$ and $x' \leq x$ implies $x' \in Z$. 

\begin{lemma}
[Quillen's Product Fiber Lemma, Proposition 1.7 \cite{quillenHomotopyPropertiesPoset1978}]
\label{lemma:QuillenProductFiber}
Let $P$ and $Q$ be posets, and let $Z$ be a downward closed subset of $P \times Q$. Let $\pi_P: Z \to P$ and $\pi_Q: Z \to Q$ be the maps induced by the projection maps. If $Z_q = \{p \in P \, | \, (p, q )\in Z \}$ is contractible for all $q \in Q$, then $\|\pi_Q\|: \|\Delta Z\| \to \|\Delta Q\|$ is a homotopy equivalence.
\end{lemma}

\subsection{CW posets}
In Section~\ref{sec:short_new_proofs}, we will work with a specific type of poset called a CW poset, which allows us to discuss face posets of CW complexes that are not simplicial complexes, such as complexes obtained by taking products of simplicial complexes.

\begin{definition}[ \cite{bjornerPosetsRegularCW1984} ]
A poset $P$ is said to be a \emph{CW poset} if 
\begin{itemize}\setlength\itemsep{0.1em}
    \item $P$ has a least element $\hat{0}$,
    \item $P$ is nontrivial, i.e., has more than one element, and
    \item for all $x \in P \setminus \{ \hat{0} \}$, the open interval $(\hat{0}, x) = \{ p \in P \, | \, \hat{0} < p < x \}$ has an order complex that is homeomorphic to a sphere. 
\end{itemize}
\end{definition}

Face posets of simplicial complexes, in particular, are CW posets after augmenting with the least element $\hat{0}$. Given two CW posets $P$ and $Q$, the product poset $P \times Q$ is a CW poset \cite{bjornerPosetsRegularCW1984}. 

Given a CW complex $K$ and its cell decomposition, define its face poset $P_K$ to be the set of all closed cells ordered by containment. Note that in \cite{bjornerPosetsRegularCW1984}, this $P_K$ is augmented with a least element $\hat{0}$.

In general, the order complex $\Delta P_K$ doesn't reveal the topology of $K$. Recall that a CW complex is called \emph{regular} if all closed cells are homeomorphic to closed balls $E^n$.  
When $K$ is a regular CW complex, such as a simplicial complex, then  $\Delta P_K $ is homeomorphic to $K$ \cite{bjornerPosetsRegularCW1984}. Thus, regular CW complexes are extremely nice in that the incidence relations of cells determine the topology. In fact, the order complex $\Delta P_K $ provides a subdivision of the regular CW complex (\cite{lundellTopologyCWComplexes1969} proof of Theorem 1.7, p.80).

\begin{lemma} 
[\cite{bjornerPosetsRegularCW1984}]
\label{lemma:CW_poset}
A poset $P$ is a CW poset if and only if it is isomorphic to the face poset of a regular CW complex. 
\end{lemma}

For example, let $P$ be the poset illustrated in Figure~\ref{fig:CW_poset} (left). The poset, augmented with $\hat{0}$, is a CW poset as it is the face poset of a regular CW complex $K$ shown in Figure~\ref{fig:CW_poset} (center). The order complex $\Delta P$, shown in Figure~\ref{fig:CW_poset} (right), is a subdivision of the CW complex.

\begin{figure}[h!]
    \centering
    \includegraphics[scale=0.3]{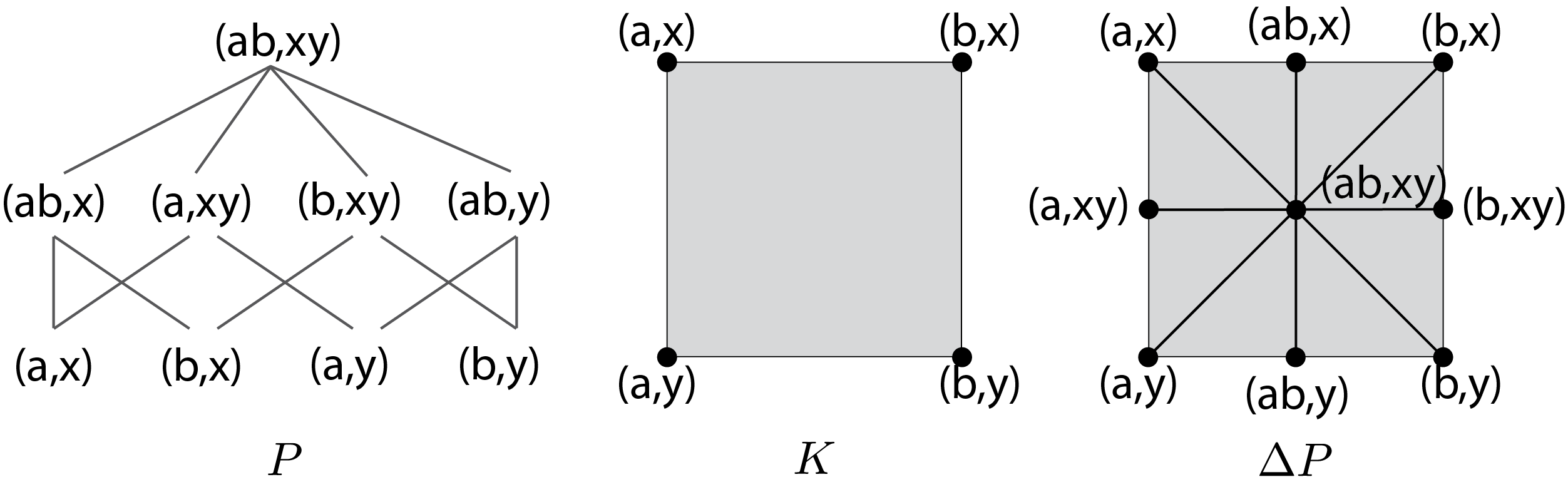}
    \caption{Example CW poset. }
    \label{fig:CW_poset}
\end{figure}

\section{Short, new proofs of Dowker duality}
\label{sec:short_new_proofs}
We now present three short new proofs of Dowker duality. In Section~\ref{sec:Dowker_Galois}, we use Galois connections. In Section~\ref{sec:Dowker_relational_join}, we introduce the relational join complex and prove Dowker duality. In Section~\ref{sec:Dowker_relational_product}, we introduce the relational product complex, which is a modification of products of CW complexes, and prove Dowker duality. Finally, in Section~\ref{section:nerve_lemma_pf_via_relational_complexes}, we present proofs of the Nerve Lemma using relational join and relational products. 

Before proceeding with the proofs, let us clarify the notions of relations between sets, posets, and simplicial complexes. As before, a relation between sets $A$ and $X$ is $R \subseteq A \times X$. 

\begin{definition}
\label{def:poset_relations}
A \emph{relation between posets} $P$ and $Q$ is a downward closed subset $\tilde{R} \subseteq P \times Q$. That is, if $p\tilde{R}q$, then $p'\tilde{R}q'$ for all $p' \leq p$ and $q' \leq q$. 

\end{definition}

Given a Galois connection, one can define a relation between the underlying posets as follows. 

\begin{definition}
\label{def:GC_induced_relation}
Let $L: P \leftrightarrows Q: U$ be a Galois connection. Let $\tilde{R} \subseteq P \times Q^{\op}$ be a relation between posets defined by $p\tilde{R}q$ if and only if $Lp \leq q$ in $Q$ (equivalently, $p \leq Uq$ in $P$). We call $\tilde{R}$ the \emph{relation induced by} $(L, U)$. 
\end{definition}

A relation between posets induces a relation between the order complexes as follows.

\begin{definition}
\label{def:poset_relation_induced_ordercomplex_relation}
Let $\tilde{R} \subseteq P \times Q$ be a relation between posets. 
Let $R^*$ be a relation between the face posets of the order complexes defined as follows:  If $\sigma_P = (p_0 < \cdots < p_n)$ is a simplex of $\Delta P$ and $\sigma_Q = (q_0 < \cdots < q_m)$ is a simplex of $\Delta Q$, then $\sigma_P R^* \sigma_Q$ if and only if $p_n \tilde{R} q_m$. We call $R^*$ the \emph{relation induced by} $\tilde{R}$.
\end{definition}
Since $\tilde{R}$ is downward closed, we have that all $p_0, \dots, p_n$ are related to all $q_0, \dots, q_m$ in the above definition.

Lastly, a relation between sets induces a relation between the face posets of Dowker complexes.
\begin{definition}
\label{def:dowker_relations_induced_from_sets}
Let $R \subseteq A \times X$ be a relation between sets. 
Let $P_A, P_X$ denote the face posets of Dowker complexes $D_A$ and $D_X$. Let $\tilde{R} \subseteq P_A \times P_X$ be a relation between the face posets of Dowker complexes defined by $\sigma_A \tilde{R} \sigma_X$ if and only if $aRx$ for all $a \in \sigma_A$ and $x \in \sigma_X$. We call $\tilde{R}$ the \emph{relation induced by} $R$.
\end{definition}

\subsection{Dowker duality via Galois connections}
\label{sec:Dowker_Galois}
We present the first short proof of Dowker duality using Galois connections. It uses the fact that a relation $R \subseteq A \times X$ induces a Galois connection between the power sets of $A$ and $X$. Because the simplices in a Dowker complex consist of non-empty, finite subsets, this particular proof is limited to relations $R \subseteq A \times X$ where every element is related to a finite number of elements. The two other proofs presented in Sections~\ref{sec:Dowker_relational_join} and \ref{sec:Dowker_relational_product} do not have this restriction.

\begin{lemma} Let $R \subseteq A \times X$ be a relation between sets such that every element is related to a finite number of elements. Let $P_A$ and $P_X$ be the face posets of the Dowker complexes $D_A$ and $D_X$. Define $L: P_A \to P_X^{\op}$ by $L(\sigma_A) = \{x \in X \; | \; aRx \; \forall a \in \sigma_A \}$, and define $U: P_X^{\op} \to P_A$ by $U(\sigma_X) = \{a \in A \; | \; aRx \; \forall x \in \sigma_X \}.$ The morphisms $L$ and $U$ form a Galois connection $L: P_A \leftrightarrows P_X^{\op}: U$. 
\end{lemma}
We refer to the resulting Galois connection  $L: P_A \leftrightarrows P_X^{\op}: U$ as the \emph{Galois connection induced by R}.

\begin{theorem}
[Dowker duality via Galois connections]
\label{thm:dowker_Galois}
Let $R \subseteq A \times X$ be a relation between sets such that every element is related to a finite number of elements. 
Then $\|D_A\| \htpyeq \|D_X\|$.
\end{theorem}

\begin{proof}

Let $L: P_A \leftrightarrows P_X^{\op}: U$ be the Galois connection induced by $R$. By Lemma~\ref{lemma:Galois_htpyeq}, $|| \Delta P_A || \htpyeq || \Delta P_X||$. Note that $\Delta P_A$ is the barycentric subdivision of $D_A$, and $\Delta P_X$ is the barycentric subdivision of $D_X$. Thus, $\|D_A\| \htpyeq \|D_X\|$.

\end{proof}

\begin{example}
\label{example:running}
\normalfont
Consider the following example from \cite{brunRectangleComplexRelation2022}.
Let $A = \{ a, b, c, d\}$ and $X = \{ w, x, y, z \}$, and let $R$ be the relation  
\[R = 
\begin{blockarray}{ccccc}
 & w & x & y & z \\
\begin{block}{c[cccc]}
a & 0 & 1& 0 & 1\bigstrut[t] \\
b & 0 & 1 & 1 & 0  \\
c & 0 & 0 & 1 & 1\\
d & 1 & 0& 1 & 0 \bigstrut[b]\\
\end{block}
\end{blockarray} \, .
\]
As expected, $\|D_A\| \htpyeq \|D_X\|$ (Figure~\ref{fig:example_GC}A). Figure~\ref{fig:example_GC}B shows the face posets $P_A$ and $P_X$ of the Dowker complexes. 
Figure~\ref{fig:example_GC}C illustrates the Galois connection $L: P_A \leftrightarrows P_X^{\op}: U$.

\begin{figure}[h!]
    \centering
    \includegraphics[width=\linewidth]{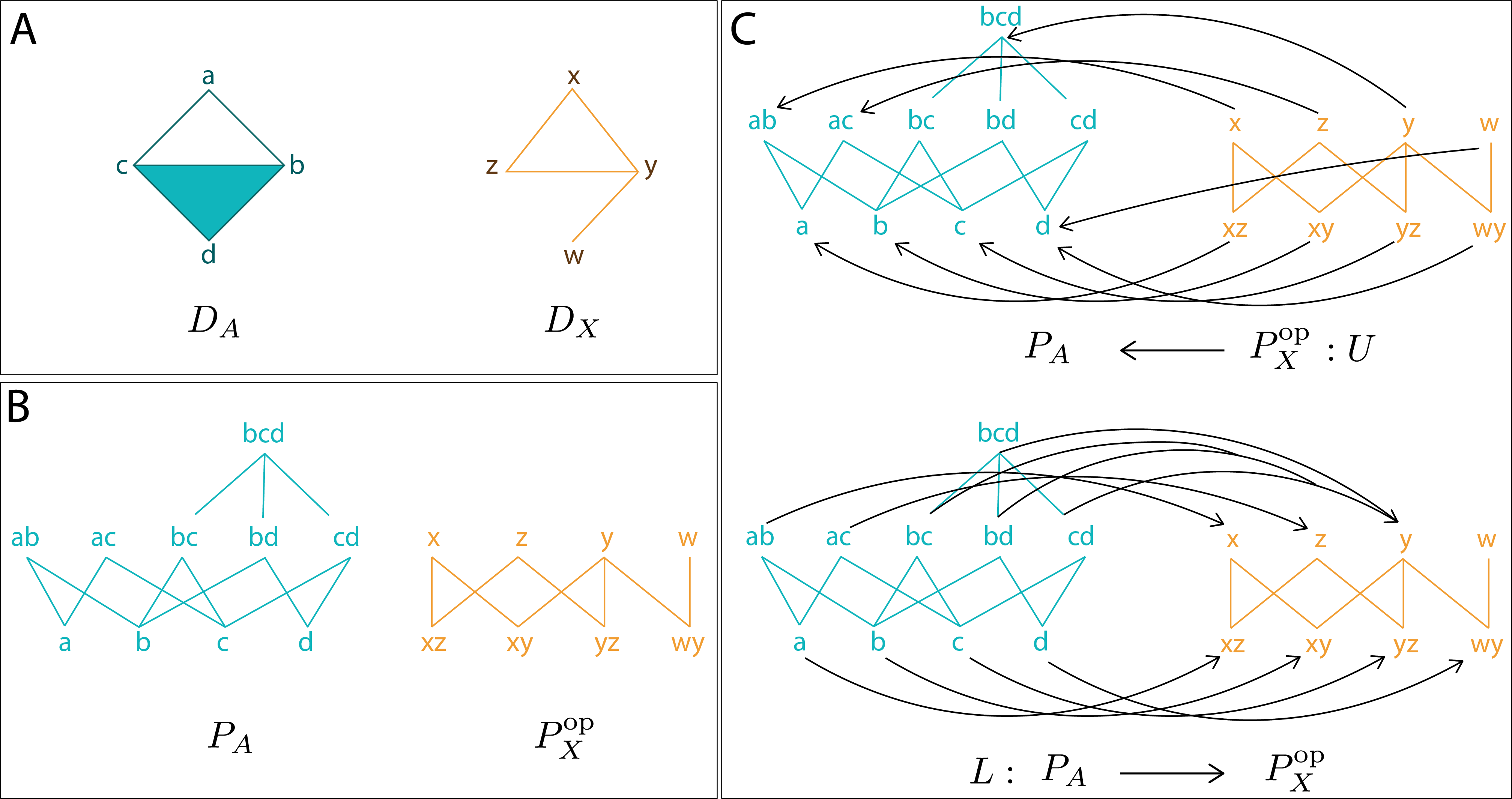}
    \caption{Example Galois connection on the face posets of Dowker complexes corresponding to Example~\ref{example:running}. \textbf{A.} Dowker complexes $D_A$ and $D_X$. \textbf{B.} Face posets $P_A$, $P_X^{\op}$ of the Dowker complexes $D_A$ and $D_X$. \textbf{C.} Illustration of the Galois connection $L: P_A \leftrightarrows P_X^{\op}: U$. }
    \label{fig:example_GC}
\end{figure}

\end{example}

We now prove the Functorial Dowker duality using Galois connections.

\begin{theorem}
[Functorial Dowker duality via Galois connections]
\label{thm:functorialDowker_filteredGC}
Let $R \subseteq A \times X$ and $R' \subseteq A' \times X'$ be relations between sets where every element is related to a finite number of elements. 
Let $f: A \to A'$ and $g: X \to X'$ be set maps such that $(f(a), g(x)) \in R'$ for all $(a, x) \in R$. Let $D_f: D_A \to D_{A'}$ and $D_g: D_X \to D_{X'}$ be the simplicial maps induced by the maps $f$ and $g$ on the Dowker complexes. There exist homotopy equivalences $\Gamma: \|D_A\| \to \|D_X\|$ and $\Gamma': \|D_{A'}\| \to \|D_{X'}\|$ such that the following diagram commutes up to homotopy.
\begin{equation} 
\label{diagram:functorial_Dowker_GC}
\begin{tikzcd}
\|D_A\| \arrow[r, "\|D_f\|"] \arrow[d, "\Gamma"] & \|D_{A'}\| \arrow[d, "\Gamma'"]\\
\|D_{X}\| \arrow[r, "\|D_g\|"] & \|D_{X'}\|
\end{tikzcd}
\end{equation}
\end{theorem}

\begin{proof}
Let $P_A, P_{A'}, P_X$, and $P_{X'}$ denote the face posets of the Dowker complexes. Let $L: P_A \leftrightarrows P_X^{\op}: U$ and $L': P_{A'} \leftrightarrows P^{\op}_{X'}: U'$ denote the Galois connections induced by the relations $R$ and $R'$. 

Let $\alpha: P_A \to P_{A'}$ and $\beta: P^{\op}_X \to P^{\op}_{X'}$ be maps of posets induced by $f$ and $g$ on the face posets of the Dowker complexes. Explicitly, given a simplex $a_0 \dots a_k \in P_A$, we have $\alpha(a_0 \dots a_k) = f(a_0) \dots f(a_k)$. Similarly, $\beta(x_0 \dots x_{m}) = g(x_0) \dots g(x_m)$.

Consider the diagram
\[
\begin{tikzcd}
P_A  \arrow{r}{\alpha}  \arrow[d, "L"] & P_{A'} \arrow[d, "L'"] \\
P^{\op}_X \arrow[r, "\beta", swap] & P^{\op}_{X'}   \,. 
\end{tikzcd}
\]
Let $\sigma_A \in P_A$. Then 
\begin{align*}
\beta \circ L (\sigma_A) &= \beta(\sigma_X), \quad \text{ where } \sigma_X = \{ x \in X\, | \, aRx \  \forall a \in \sigma_A \} \\
&= \{ g(x) \in X' \, | \, aRx \ \forall a \in \sigma_A \}.
\end{align*}
On the other hand, ${L' \circ \alpha(\sigma_A) = \{x'  \in X' \, | \, f(a)R'x'  \   \forall a \in \sigma_A \}}$. Then, $L' \circ \alpha(\sigma_A) \leq \beta \circ L(\sigma_A) $ in $P_{X'}^{\op}$ for all $\sigma_A \in P_A$. By the Order Homotopy Lemma (Lemma \ref{lemma:order_htpy_thm}), it follows that $\|L'\| \circ \|\alpha\|  \htpyeq \|\beta\| \circ \|L\|$ in the following diagram. 
\[
\begin{tikzcd}
\Vert \Delta P_A \Vert \arrow{r}{\|\alpha\|} \arrow[d, swap, "\|L\|"] & \Vert \Delta P_{A'} \Vert \arrow[d, swap, "\|L'\|"] \\
\Vert \Delta P^{\op}_X \Vert \arrow[r,swap, "\|\beta\|"] & \Vert \Delta P^{\op}_{X'} \Vert 
\end{tikzcd} 
\]
Note that $\|L \|$ and $\|L'\|$ are homotopy equivalences. Appending the diagram with the usual homeomorphism between the Dowker complexes and their barycentric subdivisions, we obtain the following diagram. 
\[
\begin{tikzcd}
\|D_A\| & \arrow[l, "\cong", swap] \Vert \Delta P_A \Vert \arrow{r}{\|\alpha\|} \arrow[d, swap, "\|L\|"] & \Vert \Delta P_{A'} \Vert \arrow[d, swap, "\|L'\|"] \arrow[r, "\cong"] & \|D_{A'}\|\\
\|D_X \| & \arrow[l, "\cong"] \Vert \Delta P^{\op}_X \Vert \arrow[r,swap, "\|\beta\|"] & \Vert \Delta P^{\op}_{X'} \Vert  \arrow[r, "\cong", swap] & \| D_{X'}. \|
\end{tikzcd} 
\]
Considering the inverse maps of appropriate homeomorphisms, we obtain Diagram~\ref{diagram:functorial_Dowker_GC}, which commutes up to homotopy.
\end{proof}

\begin{remark}
\normalfont
Let us clarify a subtle difference between Dowker duality and the Nerve Lemma. In the context of Dowker duality, a relation between sets $R \subseteq A \times X$ in which every element is related to a finite number of elements induces a Galois connection $L: P_A \leftrightarrows P^{\op}_X: U$ between the face posets of Dowker complexes. 
For contrast, let $K$ be a simplicial complex and let $\mathcal{U}$ be a cover of $K$. Let $P_K$ denote the face poset of $K$, and let $P_{N\mathcal{U}}$ denote the face poset of the nerve $N\mathcal{U}$. Then, there is a relation $\tilde{R} \subseteq P_K \times P_{N\mathcal{U}}$ where $\sigma \tilde{R} \tau$ if $\sigma$ is covered by the cover elements corresponding to $\tau$. However, we don't necessarily have a Galois connection between $P_K$ and $P_{N\mathcal{U}}^{\op}$, even when $\mathcal{U}$ is a good cover. See Figure~\ref{fig:nerve_poset} for an example. Here, $K$ is the simplicial complex (teal), and $\mathcal{U}$ is a cover consisting of one cover element $x$ (orange). Figure~\ref{fig:nerve_poset}B illustrates the posets $P_K$ and $P_{N\mathcal{U}}$. The only possible map $L: P_K \to P_{N\mathcal{U}}^{\op}$ is the constant map to $x$, and $Lp \leq x$ for all $p \in P_K$. However, there is no map $U: P_{N\mathcal{U}}^{\op} \to P_K$ so that 
$L: P_K \leftrightarrows P^{\op}_{N\mathcal{U}}: U$ form a Galois connection. When $\mathcal{U}$ is a good cover, the fibers $L^{-1}_{\leq q}$ will be contractible by assumption. However, the fibers won't necessarily have minimum or maximum elements, since $L$ isn't necessarily a part of a Galois connection. 

\end{remark}

\begin{figure}[h]
    \centering
    \includegraphics[width=0.5\linewidth]{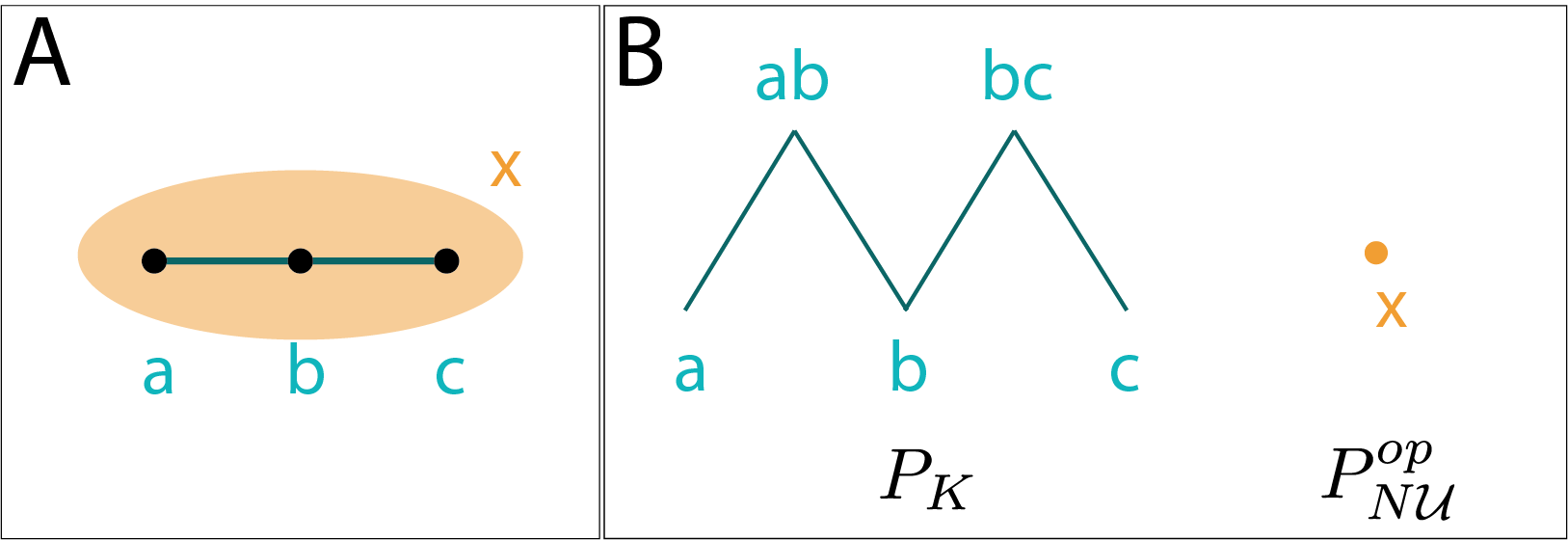}
    \caption{Example illustrating the lack of Galois connection in the context of Nerve Lemma. \textbf{A.} Simplicial complex $K$ (teal) and a cover $\mathcal{U} = \{ x \}$ (orange). \textbf{B.} The face posets $P_K$ and $P_{N\mathcal{U}}^{\op}$. There can be no Galois connection $L: P_K \leftrightarrows P_{N\mathcal{U}}^{\op}: U$. }
    \label{fig:nerve_poset}
\end{figure}

\subsection{Dowker duality via relational join}
\label{sec:Dowker_relational_join}
We now introduce the relational join complex and present a second new proof of Dowker duality. %

\begin{definition}
\label{def:relational_join}
Let $K, M$ be simplicial complexes with disjoint vertex sets, and let $\tilde{R} \subseteq P_K \times P_M$ be a relation between the face posets.
Let $K \star_{\tilde{R}} M$ be the abstract simplicial complex
\[ K \star_{\tilde{R}} M = K \cup M \cup \{\sigma_K \cup \sigma_M  \, | \, \sigma_K \tilde{R} \sigma_M \}. \]
We call $K \star_{\tilde{R}} M$ the \emph{relational join complex}. 
\end{definition}

Note that $K \star_{\tilde{R}} M$ is a subcomplex of the join $K \star M$. Furthermore, there is an inclusion of simplicial complexes $K \hookrightarrow K \star_{\tilde{R}} M$ and $M \hookrightarrow K \star_{\tilde{R}} M$. 

Given a relation $R \subseteq A \times X$ between sets, we can construct a relational join complex from the two Dowker complexes.

\begin{definition}
\label{def:Dowker_join_complex}
Let $R \subseteq A \times X$ be a relation between sets and 
let $\tilde{R} \subseteq P_A \times P_X$ be the induced relation on the face posets $P_A$, $P_X$ of Dowker complexes $D_A$ and $D_X$. 
(Definition~\ref{def:dowker_relations_induced_from_sets}). We refer to the relational join complex $D_A \star_{\tilde{R}} D_X$ as the \emph{Dowker join complex}. Explicitly, 
\[D_A \star_{\tilde{R}} D_X  =D_A \cup D_X \cup \{\sigma_A \cup \sigma_X \, | \, \sigma_A \in D_A, \; \sigma_X \in D_X, \; aRx \quad \forall a \in \sigma_A, \,  \forall x \in \sigma_X\}.\]
\end{definition}

The Dowker join complex first appeared in the following preprint as the biclique complex \cite{brun_dowker_2024}. There, the authors prove Dowker duality using discrete Morse theory.

\begin{example}
\label{DowkerJoinExample2}
\normalfont
Recall Example~\ref{example:running} with relation
\[
R = \begin{blockarray}{ccccc}
 & w & x & y & z \\
\begin{block}{c[cccc]}
a & 0 & 1& 0 & 1\bigstrut[t] \\
b & 0 & 1 & 1 & 0  \\
c & 0 & 0 & 1 & 1\\
d & 1 & 0& 1 & 0 \bigstrut[b]\\
\end{block}
\end{blockarray} \, .
\]
Figure~\ref{fig:relational_join2} illustrates the Dowker complexes $D_A, D_X$ and the Dowker join complex $D_A \star_{\tilde{R}} D_X$. 

\begin{figure}[h!]
    \centering
    \includegraphics[scale=0.25]{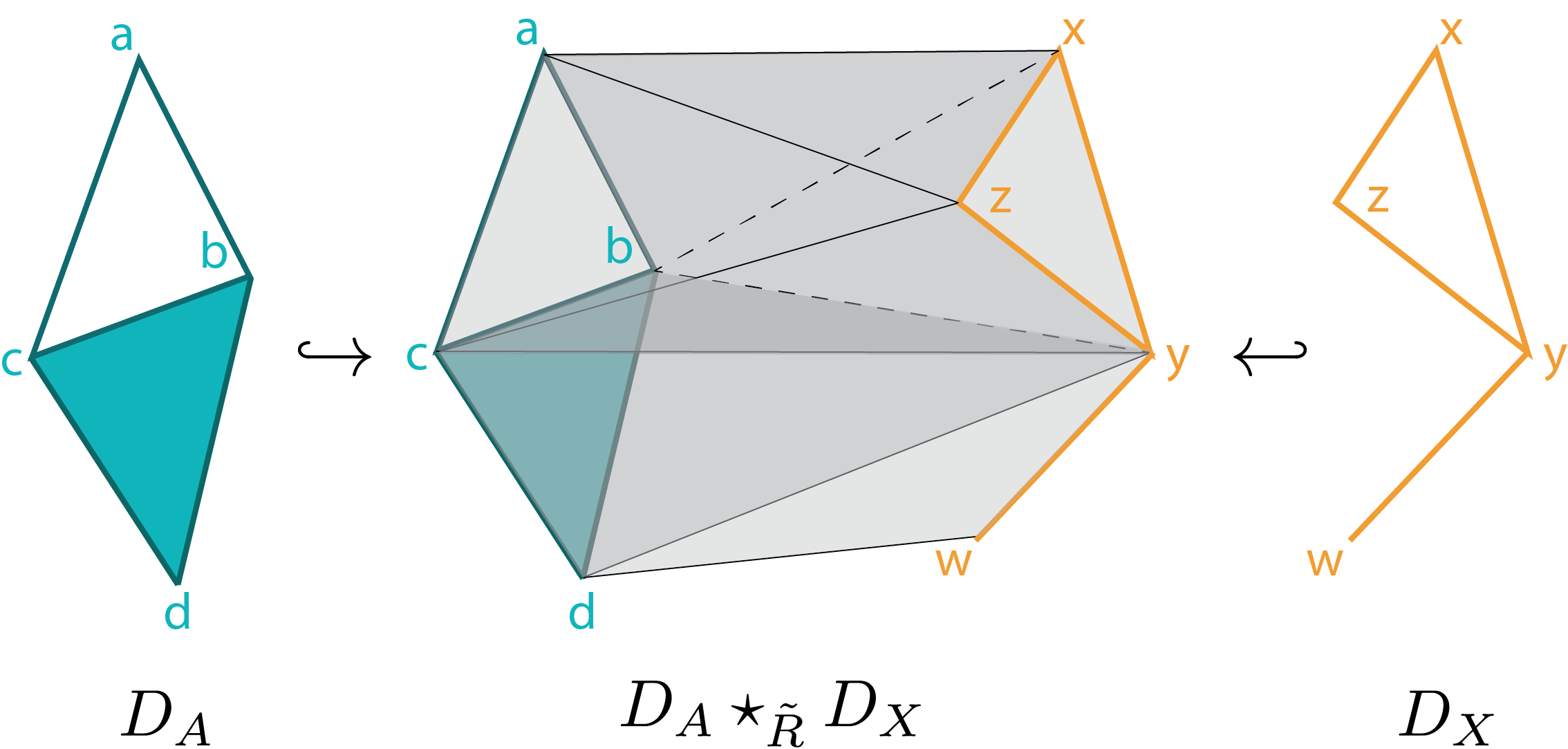}
    \caption{Dowker join complex corresponding to Example~\ref{DowkerJoinExample2}.}
    \label{fig:relational_join2}
\end{figure}

\end{example}

Before establishing a homotopy equivalence between the Dowker join complex and the Dowker complexes, we define an analogous concept for posets.

\begin{definition} (\cite{baclawskiGaloisConnectionsLeray1977, walkerHomotopyTypeEuler1981})
Let $P$ and $Q$ be disjoint posets. Given a relation $\tilde{R} \subseteq P \times Q$ of posets, define a \emph{relational join poset}, denoted $P\star_{\tilde{R}} Q^{\op}$ to be the poset whose underlying set is the disjoint union of $P$ and $Q^{\op}$. The ordering is given by the following:
\begin{enumerate}
\item if $p, p' \in P$: \quad $p \leq p'$ in $P\star_{\tilde{R}} Q^{\op}$ if and only if $p \leq p'$ in $P$ 
\item if $q, q' \in Q^{\op}$: \quad $q \leq q'$ in $P\star_{\tilde{R}} Q^{\op}$ if and only if $q \leq q'$ in $Q^{\op}$ (so $q' \leq q$ in $Q$) 
\item if $p \in P$ and $q \in Q^{\op}$: \quad $p \leq q$ in $P\star_{\tilde{R}}Q^{\op}$ if and only if $p\tilde{R}q$.   
\end{enumerate}
\end{definition}

\medskip

\begin{lemma}
\label{lemma:relational_join_poset_complex}
Given a relation $\tilde{R} \subseteq P \times Q$ of posets, let $R^*$ be the induced relation on the face posets of the order complexes 
(Definition~\ref{def:poset_relation_induced_ordercomplex_relation}). Then, 
\[\Delta(P \star_{\tilde{R}} Q^{\op}) = \Delta P \star_{R^*} \Delta Q.\]
\end{lemma}

\begin{proof}
There are three types of chains in $P \star_{\tilde{R}} Q^{\op}$: chains in $P$, chains in $Q$, and chains of the form $p_0 < \cdots < p_n < q_m < \cdots < q_0$. There are also three types of simplices in $\Delta P \star_{R^*} \Delta Q$: the simplices in $\Delta P$, simplices in $\Delta Q$, and simplices of the form $\sigma_P \cup \sigma_Q$ for $\sigma_P \in \Delta P$ and $\sigma_Q \in \Delta Q$. By construction, there is a bijection between the chains of $P \star_{\tilde{R}} Q^{\op}$ and the simplices of $\Delta P \star_{R^*} \Delta Q$ respecting the three types. In particular, a chain $p_0 < \cdots < p_n < q_m < \cdots < q_0$ in  $P \star_{\tilde{R}} Q^{\op}$ corresponds to the simplex $(p_0 < \cdots < p_n) \cup (q_0 < \cdots < q_m)$ in $\Delta P \star_{R^*} \Delta Q$ and vice versa. 
\end{proof}

\begin{example}
\normalfont
Recall Example \ref{DowkerJoinExample2}. Figure ~\ref{fig:join_example2}A shows the face posets of the Dowker complexes, and Figure~\ref{fig:join_example2}B shows the relational join poset $P_A \star_{\tilde{R}} P_X^{\op}$, highlighting a specific chain that consists of elements in $P_A$ and $P_X^{\op}$. Figure~\ref{fig:join_example2}C shows the order complex of the relational join poset.
\begin{figure}[h!]
    \centering
    \includegraphics[scale=0.25]{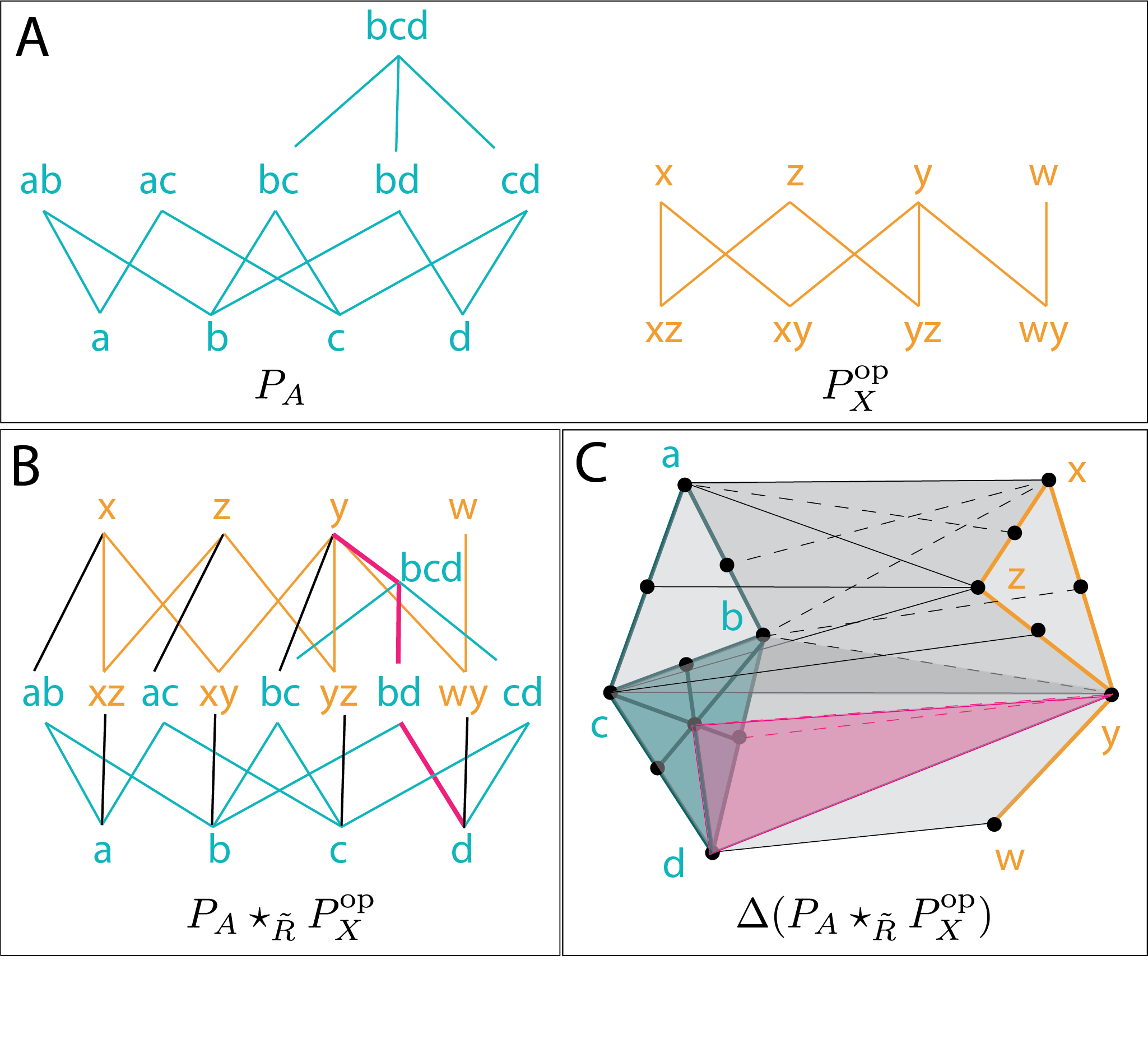}
    \caption{Illustrations of the relational join poset and its order complex. \textbf{A}. Face posets of Dowker complexes corresponding to Example~\ref{example:running}. \textbf{B}. The relational join poset. The chain $d < bd < bcd < y$ is highlighted in red. \textbf{C}. The order complex of the relational join poset. The 3-simplex highlighted in pink corresponds to the highlighted chain in panel B.}
    \label{fig:join_example2}
\end{figure}
\end{example}

We now prove Dowker duality using the relational join complex.

\begin{theorem}
[Dowker duality via relational join]
\label{thm:Dowker_relational_join}
Let $R \subseteq A \times X$ be a relation between sets, and 
let $\tilde{R} \subseteq P_A \times P_X$ be the induced relation between the face posets $P_A$, $P_X$ of Dowker complexes $D_A$ and $D_X$ (Definition~\ref{def:dowker_relations_induced_from_sets}). Then,
\[\|D_A\| \htpyeq \|D_A \star_{\tilde{R}} D_X\| \htpyeq \|D_X\|. \]

\end{theorem}

The proof involves establishing the contractibility of relevant fibers and using Quillen's Poset Fiber Lemma. 

\begin{proof}
We will show that the inclusion of posets
\[P_A \xrightarrow{i} P_A \star_{\tilde{R}} P_X^{\op} \xleftarrow{j} P_X^{\op} \]
induce homotopy equivalences in the order complexes
\begin{equation}
\label{eq:dowker_relationaljoin_htpyeq}
\| \Delta P_A\| \xrightarrow{\htpyeq} \|\Delta (P_A \star_{\tilde{R}} P_X^{\op})\| \xleftarrow{\htpyeq} \|\Delta P_X^{\op}\|.     
\end{equation}

Consider the fibers of $i: P_A \to P_A \star_{\tilde{R}} P_X^{\op}$ of the form $i^{-1}_{\leq \sigma}$. Given $\sigma_A \in P_A$, considered as an element of $P_A \star_{\tilde{R}} P_X^{\op}$, then $i^{-1}_{\leq \sigma_A}$ has a maximum element, namely $\sigma_A$. So $i^{-1}_{\leq \sigma_A}$ is contractible. If $\sigma_X \in P_X^{\op}$, then 
\begin{equation}
\label{eq:fiber_contractible}
\begin{split}
i^{-1}_{\leq \sigma_X} &= \{  \sigma_A \in P_A \, | \, \sigma_A \leq \sigma_X \text{ in } P_A \star_{\tilde{R}} P_X^{\op} \} \\
&= \{ \sigma_A \in P_A \, | \, \sigma_A \tilde{R} \sigma_X\} \\
&= \{ \sigma_A \in P_A \, | \, aRx  \quad \forall  a \in \sigma_A \text{ and } \forall x \in \sigma_X\}.
\end{split}
\end{equation}
Let $A^* \subseteq A$ be the subset $A^* = \{ a \in A \, |  aRx \, \, \forall x \in \sigma_X \}$. Then, $i^{-1}_{\leq \sigma_X}$ is the poset consisting of all non-empty finite subsets of $A^*$, which is contractible. By Quillen's Poset Fiber Lemma  (Lemma~\ref{lemma:Quillen_fiber}), $i$ induces a homotopy equivalence in the order complexes. A similar argument shows that the fibers of $j$ of the form $j^{-1}_{\geq \sigma}$ are contractible. 

Note that $\Delta P_A$ is the barycentric subdivision of $D_A$ and $\Delta P_X^{\op}$ is the barycentric subdivision of $D_X$. From Lemma~\ref{lemma:relational_join_poset_complex}, $\Delta (P_A \star_{\tilde{R}} P_X^{\op}) = \Delta P_A \star_{R^*} \Delta P_X$. Furthermore, $\Delta(P_A \star_{\tilde{R}} P_X^{\op})$ is a subdivision of the Dowker join complex $D_A \star_{\tilde{R}} D_X$ obtained by performing barycentric subdivision on simplices $s$ of $D_A$ and $D_X$ and subdividing all simplices that contain $s$ accordingly. Then, the piecewise linear map $\psi: \|\Delta(P_A \star_{\tilde{R}} P_X^{\op})\| \to \| D_A \star_{\tilde{R}} D_X \|$ mapping each vertex of $ \Delta(P_A \star_{\tilde{R}} P_X^{\op})$ to the corresponding point in $\| D_A \star_{\tilde{R}} D_X \|$ is a homeomorphism (\cite{Spanier1966} Theorem 3.3.4). It follows from Equation~\ref{eq:dowker_relationaljoin_htpyeq} that
\[\|D_A\| \htpyeq \|D_A \star_{\tilde{R}} D_X\| \htpyeq \|D_X\|.\]
\end{proof}
We now prove the functorial Dowker duality via relational joins.
\begin{theorem}[Functorial Dowker duality via relational join]
Let $R \subseteq A \times X$ and $R' \subseteq A' \times X'$ be relations between sets. Let $f: A \to A'$ and $g: X \to X'$ be set maps such that $(f(a), g(x)) \in R'$ for all $(a, x) \in R$. Let $\tilde{R} \subseteq P_A \times P_X$ and $\tilde{R}' \subseteq P_{A'} \times P_{X'}$ be the induced relation between the face posets of the Dowker complexes. Let $D_f: D_A \to D_{A'}$, $D_g: D_X \to D_{X'}$, and $D_{f \star g}: D_A \star_{\tilde{R}} D_X \to D_{A'} \star_{\tilde{R}'} D_{X'}$ be the simplicial maps induced by the maps $f$ and $g$ on the vertex sets. 
Explicitly, 
\[
D_{f \star g}  (\sigma) = \begin{cases}
f(a_0) \dots f(a_n) \quad \text{ if } \sigma = a_0 \dots a_n  \in D_A \\
g(x_0) \dots g(x_m) \quad \text{ if } \sigma = x_0 \dots x_m \in D_X \\
f(a_0) \dots f(a_n) g(x_0) \dots g(x_m) \quad \text{ if } \sigma = a_0 \dots a_n x_0 \dots x_m. 
\end{cases}
\] 
Then, the following diagram commutes up to homotopy.

\begin{equation}
\label{eq:funct_join}
\begin{tikzcd}
\|D_A\|  \arrow[d, "\|D_f\|"] \arrow{r}{\htpyeq} & \|D_A \star_{\tilde{R}} D_X\| \arrow[d, "\|D_{f \star g}\|"] & \|D_X\|  \arrow[l, "\htpyeq", swap] \arrow[d, "\|D_g\|"]  \\
\|D_{A'}\| \arrow{r}{\htpyeq} & \|D_{A'} \star_{\tilde{R}} D_{X'}\| & \arrow[l, "\htpyeq", swap] \|D_{X'}\|
\end{tikzcd}
\end{equation}

\end{theorem}

\begin{proof}
Let us focus on the left square. Let $\tilde{f}: P_A \to P_{A'}$ be the induced map on the face posets of Dowker complexes defined by $\tilde{f}(a_0 \dots a_n) = f(a_0) \dots f(a_n)$. Let $\tilde{g}: P^{\op}_X \to P^{\op}_{X'}$ be defined similarly. Let $\tilde{f} \star \tilde{g}: P_A \star_{\tilde{R}} P^{\op}_X \to  P_{A'} \star_{\tilde{R}} P^{\op}_{X'}$ be the induced map on the relational join poset defined by

\[ 
\tilde{f} \star \tilde{g} (\sigma) =
\begin{cases}
\tilde{f}(\sigma) \quad \text{if } \sigma \in P_A  \\
\tilde{g}(\sigma) \quad \text{if } \sigma \in P_X^{\op} 
\end{cases}
\]

We have the following commutative diagram of posets and spaces. Here, $i$ and $i'$ are inclusion of posets. 
\[
\begin{tikzcd}
P_A  \arrow[d,"\tilde{f}", swap] \arrow{r}{i} & P_A \star_{\tilde{R}} P^{\op}_X \arrow[d, "\tilde{f} \star \tilde{g}"]  \\
P_{A'} \arrow{r}{i'} & P_{A'} \star_{\tilde{R}} P^{\op}_{X'} 
\end{tikzcd}
\hspace{1cm} \text{and} \hspace{1cm}
\begin{tikzcd}
\| \Delta P_A \|  \arrow[d,"\| \tilde{f} \|", swap] \arrow{r}{\|i\|} & \| \Delta( P_A \star_{\tilde{R}} P^{\op}_X ) \|\arrow[d, " \| \tilde{f} \star \tilde{g}\| "]  \\
\| \Delta P_{A'} \| \arrow{r}{\| i'\|} & \| \Delta (P_{A'} \star_{\tilde{R}} P^{\op}_{X'}) \| 
\end{tikzcd}
\] 

Appending the right diagram with appropriate homeomorphisms $\phi, \phi', \psi,$ and $\psi'$ between geometric realizations of simplicial complexes and their subdivisions, we obtain the following diagram. 
\[
\begin{tikzcd}
\| D_A \| \arrow[d, " \| D_f \|", swap] & \arrow[l, "\phi", swap ]
\| \Delta P_A \|  \arrow[d,"\| \tilde{f} \| ", swap] \arrow{r}{\|i\|} &  \| \Delta( P_A \star_{\tilde{R}} P^{\op}_X ) \| \arrow[d, " \|  \tilde{f} \star \tilde{g} \| "] \arrow[r, "\psi"] & \|D_A \star_{\tilde{R}} D_X  \| \arrow[d, "\| D_{f \star g}\|"] 
 \\
\| D_{A'}\|  & \arrow[l, "\phi'", swap] \| \Delta P_{A'} \| \arrow{r}{\|i\|} & \| \Delta (P_{A'} \star_{\tilde{R}} P^{\op}_{X'}) \| \arrow[r, "\psi'"] & \| D_{A'} \star_{\tilde{R}} D_{X'} \|
\end{tikzcd}
\] 
The center square commutes. Note that $\tilde{f}:P_A \to P_{A'}$, when considered as a simplicial map $\tilde{f}: \Delta P_A \to \Delta P_{A'}$, is the map induced by $D_f: D_A \to D_{A'}$ on the order complex of its face posets. So the left square commutes up to homotopy by Lemma~\ref{lemma:functorial_subdivision}. Using a similar argument as Lemma~\ref{lemma:functorial_subdivision}, one can show that the right square also commutes up to homotopy.  Given a simplex 
\[\sigma = (a_0 < a_0a_1< \dots < a_0 a_1 \dots a_n < x_0x_1 \dots x_m < \dots < x_0)\]
of $\Delta(P_A \star_{\tilde{R}} P^{\op}_X)$, let 
\begin{align*}
C(\sigma)  &= \| D_{f \star g} \| ( \|a_0 \dots a_n\,x_0 \dots x_m\|) \\
& = \| f(a_0) \dots f(a_n)\, g(x_0) \dots g(x_m)\| .
\end{align*}
$C(\sigma)$ is contractible for all $\sigma \in \Delta(P_A \star_{\tilde{R}} P^{\op}_X)$ and $C$ carries $\| D_{f \star g}\|  \circ \psi $ and $\psi' \circ \| \tilde{f} \star \tilde{g} \|$. By the Carrier Lemma (Lemma~\ref{lemma:carrierlemma}), the diagram commutes up to homotopy.

A similar proof shows that the right square of Diagram~\ref{eq:funct_join} commutes up to homotopy.
\end{proof}

\subsection{Dowker duality via relational product}
\label{sec:Dowker_relational_product}
We now introduce the relational product complex and present a third proof of Dowker duality. The construction is inspired by the rectangle complex in \cite{brunRectangleComplexRelation2022}.

\begin{definition}
Let $K$, $M$ be simplicial complexes, and 
let $\tilde{R} \subseteq P_K \times P_M$ be a relation between the face posets. The \emph{relational product complex} is the CW-complex
\[ K \times_{\tilde{R}} M = \bigcup_{\sigma \tilde{R} \tau } \|  \sigma\| \times \|\tau\|. \]
\end{definition}

The relational product complex is not necessarily a simplicial complex, as illustrated in examples \ref{different_product_rectangle_example} and \ref{example:relational_product}. While products of CW complexes need not be CW complexes, the product of CW complexes is a CW complex with the compactly generated CW topology \cite{brooke-taylorProductsCWComplexes2018, }.

In fact, since $\|K\|$ and $\|M\|$ are regular CW complexes, their product is also a regular CW complex, and its subcomplex $K \times_{\tilde{R}} M$ is also a regular CW complex \cite{lundellTopologyCWComplexes1969}. 
Note that if $\sigma$ is a $p$-simplex of $K$ and $\tau$ is a $q$-simplex of $M$ that is related to $\sigma$, then $\|\sigma\| \times \|\tau\|$ is a closed $p+q$ cell of $K \times_{\tilde{R}} M$.  

Let $\tilde{R} \subseteq P_K \times P_M$ be a relation between the face posets of simplicial complexes $K$ and $M$. Then, there exists a partial order on $\tilde{R}$ coming from the product partial order on $P_K \times P_M$. 
\begin{lemma}
\label{lemma:R_faceposet_productcomplex}

Let $K, M$ be simplicial complexes, and let $\tilde{R} \subseteq P_K \times P_M$ be a relation between face posets, considered as a subposet of the product poset $P_K \times P_M$. Then, $\tilde{R}$ is the face poset of the relational product complex $K \times_{\tilde{R}} M$. 

\end{lemma}

\begin{proof}
Let $P_{\times}$ be the face poset of the relational product complex $K \times_{\tilde{R}}M$, which is the collection of closed cells ordered by containment \cite{bjornerPosetsRegularCW1984}.

Define a map of posets $\phi: \tilde{R} \to P_{\times}$ by $\phi(\sigma, \tau) = \|\sigma\| \times \|\tau\|$. One can check that $\phi$ is bijective. Note that $\phi$ is order-preserving since 
\begin{align*}
(\sigma, \tau) \leq (\sigma', \tau') \text{ in } \tilde{R} & \iff \text{$\sigma \leq \sigma'$ in $P_K$ and $\tau \leq \tau'$ in $P_M$} \\
& \iff \|\sigma\| \subseteq \|\sigma' \| \text{ in } P_K \text{ and } \| \tau \| \subseteq \|\tau'\| \text{ in } P_M\\
& \iff \text{$\| \sigma\| \times \|\tau\| \subseteq \|\sigma'\| \times \| \tau' \|$ in $K \times_{\tilde{R}} M$} \\
& \iff \text{$\| \sigma \| \times \|\tau\| \leq \| \sigma'\| \times \|\tau'\|$ in $P_\times$}.
\end{align*}
Thus, $\tilde{R}$ and $P_{\times}$ are isomorphic posets.
\end{proof}
Since $\tilde{R}$ is the face poset of $K \times_{\tilde{R}} M$, $\Vert \Delta \tilde{R} \Vert$ is homeomorphic to $ K \times_{\tilde{R}} M$ \cite{bjornerPosetsRegularCW1984}.

Given a relation $R \subseteq A \times X$ between sets, we can construct a relational product complex from the two Dowker complexes.

\begin{definition}
\label{def:Dowker_product_complex}
Let $R \subseteq A \times X$ be a relation between sets and let $\tilde{R} \subseteq P_A \times P_X$ be the induced relation on the face posets $P_A$, $P_X$ of Dowker complexes $D_A$ and $D_X$ (Definition~\ref{def:dowker_relations_induced_from_sets}). We refer to the relational product complex $D_A \times_{\tilde{R}} D_X$ as the \emph{Dowker product complex}. Explicitly,
\[D_A \times_{\tilde{R}} D_X = \bigcup_{\sigma_A \tilde{R} \sigma_X } \| \sigma_A \| \times \| \sigma_X \|. \]
\end{definition}

The Dowker product complex is closely related to the rectangle complex \cite{brunRectangleComplexRelation2022}, as illustrated in the following examples. 


\begin{example}
\label{ex:dowker_product_complex}
\normalfont
Recall the relation 
\[
R = \begin{blockarray}{ccccc}
 & w & x & y & z \\
\begin{block}{c[cccc]}
a & 0 & 1& 0 & 1\bigstrut[t] \\
b & 0 & 1 & 1 & 0  \\
c & 0 & 0 & 1 & 1\\
d & 1 & 0& 1 & 0 \bigstrut[b]\\
\end{block}
\end{blockarray}
\]
from Example~\ref{example:running}, and let $\tilde{R} \subseteq P_A \times P_X$ be the induced relation on the face posets of Dowker complexes. Here, the Dowker product complex and the rectangle complex \cite{brunRectangleComplexRelation2022} are identical (Figure~\ref{fig:relational_product_complexes}A).

\begin{figure}[h!]
    \centering
    \includegraphics[scale=0.4]{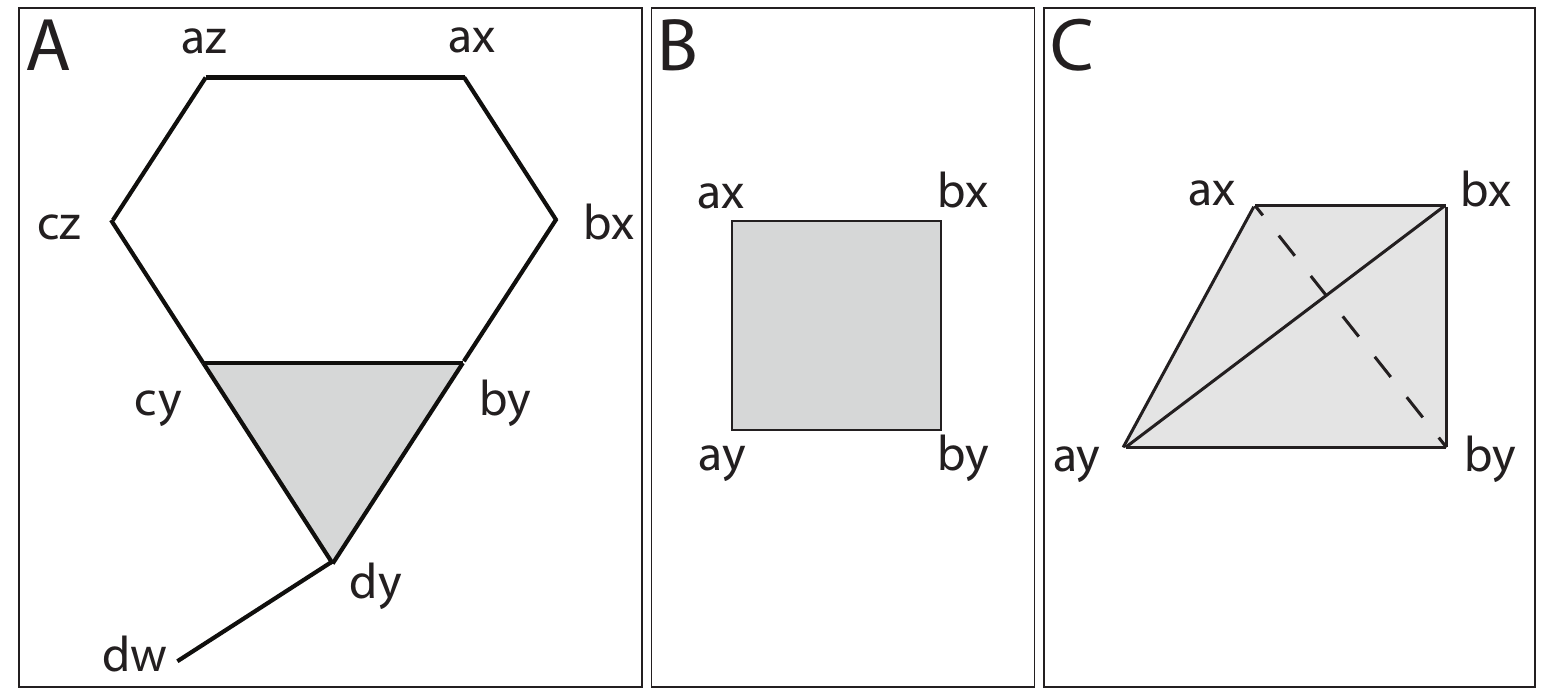}
    \caption{Example Dowker product complexes. 
    \textbf{A}.
    The Dowker product complex and the rectangle complex corresponding to the relation in Example~\ref{ex:dowker_product_complex}.
    \textbf{B}. The Dowker product complex corresponding to Example~\ref{different_product_rectangle_example}. 
    \textbf{C}. The rectangle complex corresponding to Example~\ref{different_product_rectangle_example}.  
    }
    \label{fig:relational_product_complexes}
\end{figure}

\end{example}

\begin{example}
\label{different_product_rectangle_example}
\normalfont
Here is an example where the Dowker product complex differs from the rectangle complex. Let $A= \{ a, b\},$ $X = \{x, y\}$, and let $R \subseteq A \times X$ be the relation
\[
R = \begin{blockarray}{ccc}
 & x & y  \\
\begin{block}{c[cc]}
a &  1& 1  \bigstrut[t] \\
b &  1 & 1  \bigstrut[b]\\
\end{block}
\end{blockarray}
. \]
Here, both $D_A$ and $D_X$ consist of a single $1$-simplex. Let $\tilde{R} \subseteq P_A \times P_X$ be the induced relation on the face posets of the Dowker complexes. The Dowker product complex is a CW complex with one 2-cell (Figure~\ref{fig:relational_product_complexes}B), while the rectangle complex is a simplicial complex with one 3-simplex (Figure~\ref{fig:relational_product_complexes}C).
\end{example}

\begin{example}
\label{example:relational_product}
\normalfont
Let $A = \{a, b\}$, $X = \{x, y, z \}$, and let $R \subseteq A \times X$ be the following relation. 
\[
R = \begin{blockarray}{cccc}
 & x & y & z \\
\begin{block}{c[ccc]}
a &  1& 1 & 0 \bigstrut[t] \\
b &  1 & 1 & 1 \bigstrut[b]\\
\end{block}
\end{blockarray}
. \]
Figure~\ref{fig:relational_product_example2} illustrates the Dowker product complex.

\begin{figure}[h!]
    \centering
    \includegraphics[scale=0.3]{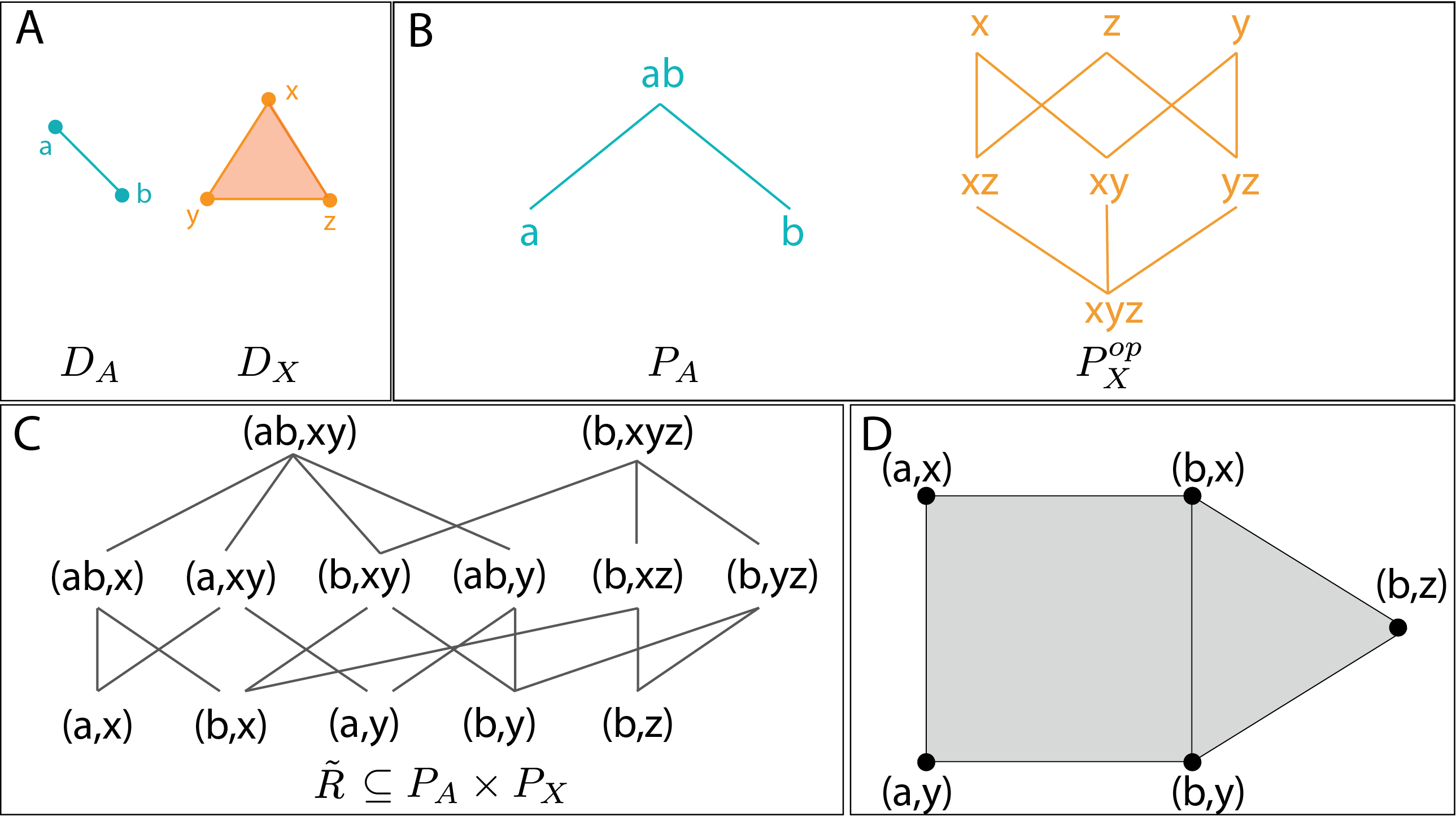}
    \caption{Example Dowker product complex. \textbf{A.} Dowker complexes for Example~\ref{example:relational_product}. \textbf{B.} Face posets of Dowker complexes. \textbf{C.} The poset $\tilde{R} \subseteq P_A \times P_X$. \textbf{D.} The Dowker product complex. }
    \label{fig:relational_product_example2}
\end{figure}

\end{example}

We now prove Dowker duality using relational products.

\begin{theorem}
[Dowker duality via relational products]
\label{thm:Dowker_relational_product}
Let $R \subseteq A \times X$ be a relation between sets, and 
let $\tilde{R} \subseteq P_A \times P_X$ be the induced relation between the face posets $P_A$, $P_X$ of Dowker complexes $D_A$ and $D_X$. 
Then, 
\[ \|D_A\|  \htpyeq  D_A \times_{\tilde{R}} D_X \htpyeq  \|D_X\| . \]  
\end{theorem}

\begin{proof} 

Considering $\tilde{R} \subseteq P_A \times P_X$ as a downward closed subset of $P_A \times P_X$, let $\pi_A: \tilde{R} \to P_A$ and $\pi_X: \tilde{R} \to P_X$ be the maps induced by the projection maps. We will show that
\[ P_A \xleftarrow{\pi_A} \tilde{R} \xrightarrow{\pi_X} P_X\]
induce homotopy equivalences
\begin{equation}
\label{eq:htpyeq_product_complex}
\| \Delta P_A \| \xleftarrow{\htpyeq} \| \Delta \tilde{R} \| \xrightarrow{\htpyeq} \|\Delta P_X\|. 
\end{equation}

Fix $\sigma_A \in P_A$ and consider 
\begin{align*}
\tilde{R}_{\sigma_A} &= \{\sigma_X \in P_X \, | \, (\sigma_A, \sigma_X) \in \tilde{R} \} \\
&= \{ \sigma_X \in P_X \, | \, (a, x) \in R \quad  \forall a\in \sigma_A, \, \forall x \in \sigma_X \}.
\end{align*}
Let $X^* \subseteq X$ be the subset $X^* = \{ x \in X \, | \, (a,x) \in R \quad \forall a \in \sigma_A \}$. Then, $\tilde{R}_{\sigma_A}$ is the collection of non-empty finite subsets of $X^*$, partially ordered by inclusion and is contractible.  By Quillen's Product Fiber Lemma (Lemma \ref{lemma:QuillenProductFiber}), the map $\pi_A$ induces a homotopy equivalence in the order complexes. A similar argument for $\pi_X$ establishes Equation~\ref{eq:htpyeq_product_complex}.

According to Lemma~\ref{lemma:R_faceposet_productcomplex}, $\tilde{R}$ is the face poset of the relational product complex $D_A \times_{\tilde{R}} D_X$.  
Since $P_A$ and $P_X$ are the face posets of the Dowker complexes, it follows from Equation~\ref{eq:htpyeq_product_complex} that 
\[ \|D_A\|  \xleftarrow{\htpyeq}  D_A \times_{\tilde{R}} D_X \xrightarrow{\htpyeq}  \| D_X\| . \]  
\end{proof}

\begin{remark}
\normalfont 
Given a relation $R \subseteq A \times X$, one can consider a directed graph $G$ with vertex set $A \cup X$ and oriented edges $(a, x)$ if $aRx$. In this perspective, the 
Dowker complexes $D_A$ and $D_X$ coincide with the in-neighborhood complex $\overleftarrow{N(G)}$ and the out-neighborhood complex $\overrightarrow{N(G)}$ in \cite{dochtermann_homomorphism_2023}. Furthermore, the relational product complex $D_A \times_{\tilde{R}} D_X$ coincides with $\overrightarrow{\text{Hom}}(K_2, G)$ in \cite{dochtermann_homomorphism_2023}, and Theorem 4.3 in \cite{dochtermann_homomorphism_2023} provides an alternative proof of Theorem~\ref{thm:Dowker_relational_product}. 
\end{remark}
We now prove the functorial Dowker duality via relational products.
\begin{theorem}[Functorial Dowker duality via relational product]
Let $R \subseteq A \times X$ and $R' \subseteq A' \times X'$ be relations between sets. Let $f: A \to A'$ and $g: X \to X'$ be set maps such that $(f(a), g(x)) \in R'$ for all $(a, x) \in R$. Let $\tilde{R} \subseteq P_A \times P_X$ and $\tilde{R'} \subseteq P_{A'} \times P_{X'}$ be the induced relation between the face posets of the  Dowker complexes.  Let $D_f: D_A \to D_{A'}$ and $D_g: D_X \to D_{X'}$ be the simplicial maps induced by the maps $f$ and $g$ on the vertex sets. Let $D_{f \times g}: D_A \times_{\tilde{R}} D_X \to D_{A'} \times_{\tilde{R'}} D_{X'}$ be the cellular map 
\[ 
D_{f \times g}( \|a_0 \dots a_n\| \times \|x_0 \dots x_m\|) = \|f(a_0) \dots f(a_n) \| \times \|g(x_0) \dots g(x_m) \|.\]
Then, the following diagram commutes up to homotopy.
\begin{equation}
\label{eq:diagram_functorial_products}
\begin{tikzcd}
\|D_A\|   \arrow[d,"\|D_f\|"]  & \arrow[l,"\htpyeq", swap] D_A \times_{\tilde{R}} D_X \arrow[d, "D_{f \times g}"]  \arrow[r, "\htpyeq"]  & \|D_X\| \arrow[d, "\|D_g\|"]  \\
\|D_{A'}\| &  \arrow[l,"\htpyeq", swap] D_{A'} \times_{\tilde{R'}} D_{X'} \arrow[r, "\htpyeq"]  &  \|D_{X'}\|
\end{tikzcd}
\end{equation}

\end{theorem}

\begin{proof}
Let us focus on the right square. Let $\tilde{f}: P_A \to P_{A'}$ and $\tilde{g}: P_X \to P_{X'}$ be the induced maps on the face posets of Dowker complexes defined by $\tilde{f}(a_0 \dots a_n) = f(a_0) \dots f(a_n)$ and $\tilde{g}(x_0 \dots x_m) = g(x_0) \dots g(x_m)$. Define $\tilde{f} \times \tilde{g}: \tilde{R} \to \tilde{R}'$ by 
\[\tilde{f} \times \tilde{g}((a_0 \dots a_n),(x_0 \dots x_m)) = ((f(a_0) \dots f(a_n)), (g(x_0) \dots g(x_m))).\]

We have the following commutative diagram of posets and spaces. 
\[ 
\begin{tikzcd}
\tilde{R} \arrow[r, "\pi"] \arrow[d, "\tilde{f} \times \tilde{g}"] 
& P_X \arrow[d, "\tilde{g}"] \\
\tilde{R'} \arrow[r, "\pi'"] & P_{X'}    
\end{tikzcd}
\hspace{2em}
\text{and}
\hspace{2em}
\begin{tikzcd}
\Vert \Delta \tilde{R} \Vert  \arrow[r, "\|\pi\|"] \arrow[d, " \| \tilde{f} \times \tilde{g} \|"]  
& \Vert \Delta P_X \arrow[d, "\| \tilde{g} \|"] \Vert \\
\Vert \Delta \tilde{R'} \Vert \arrow[r, "\|\pi'\|"] 
& \Vert \Delta P_{X'}  \Vert  
\end{tikzcd} 
\]
Append the diagram with homeomorphisms $\phi$ and $\phi'$ between CW complexes and realizations of their face posets (\cite{bjornerPosetsRegularCW1984}). We also append the diagram with homeomorphisms $\psi$ and $\psi'$ between the geometric realizations of Dowker complexes and their subdivisions. We obtain the following.
\[\begin{tikzcd}
D_A \times_{\tilde{R}} D_X   \arrow[d, "D_{f \times g}"]  & \arrow[l, "\phi"]   \Vert \Delta \tilde{R} \Vert  \arrow[r, "\|\pi\|", swap] \arrow[d, "\| \tilde{f} \times \tilde{g} \|"] 
& \Vert \Delta P_X \arrow[d, "\| \tilde{g} \|"] \Vert  \arrow[r, "\psi", swap] 
&   \Vert D_X \Vert \arrow[d, "\| D_g \|"] \\
D_{A'} \times_{\tilde{R'}} D_{X'}  
& \arrow[l, "\phi'"]  \Vert \Delta \tilde{R'} \Vert \arrow[r, "\|\pi'\|", swap]
& \Vert \Delta P_{X'} \Vert  \arrow[r, "\psi'", swap]
&  \Vert D_{X'} \Vert 
\end{tikzcd} \]

The center square commutes. Note that $\tilde{g}: P_X \to P_{X'}$, when considered as the simplicial map $\tilde{g}: \Delta P_X \to \Delta P_{X'}$, is precisely the map induced by $D_g: D_X \to D_{X'}$ on the order complex of its face posets. So the right square commutes up to homotopy by Lemma~\ref{lemma:functorial_subdivision}. The left square also commutes up to homotopy for a similar reason. Given a simplex $(\sigma_0, \tau_0) < \cdots < (\sigma_n, \tau_n)$ of $\Delta \tilde{R} $,  let
\[C((\sigma_0, \tau_0) < \cdots < (\sigma_n, \tau_n)) = D_{f \times g} ( \|\sigma_n\| \times \|\tau_n\|) = \|D_f(\sigma_n) \| \times \|D_g(\tau_n)\|,  \]
which is contractible, being a closed cell of a regular CW complex $D_{A'} \times_{\tilde{R}'} D_{X'}$. Note that $C$ carries both $D_{f \times g}\circ \phi$ and $\phi' \circ \| \tilde{f} \times \tilde{g}\| $. By the Carrier Lemma (Lemma~\ref{lemma:carrierlemma}), the left square commutes up to homotopy. All together, the right square of Diagram~\ref{eq:diagram_functorial_products} commutes up to homotopy. Similarly, the left square of Diagram~\ref{eq:diagram_functorial_products} also commutes up to homotopy. 
\end{proof}
The three new proofs of Dowker duality are summarized by the following diagrams of posets (left) and of Dowker complexes and relational complexes (right). 

\[
\begin{tikzcd}[column sep = small]
\quad & \tilde{R} \arrow[dl] \arrow[dr] & \quad \\
P_A \arrow{dr}  \arrow[rr, "L", yshift=-3, swap]  & \, & P_X^{\op} \arrow{dl}  \arrow[ll, "U", yshift = 3, swap] \\
\, & P_A \star_{\tilde{R}} P_X^{\op} & \, 
\end{tikzcd} 
\hspace{5em}
\begin{tikzcd}[column sep = small]
\quad &   D_A  \times_{\tilde{R}} D_X   \arrow[dl, swap, "\htpyeq"] \arrow{dr}{\htpyeq} & \quad \\ 
\|D_A\|  \arrow{dr}[swap]{\htpyeq}   \arrow[rr, "\htpyeq", yshift=-3, swap]  & \, &   \|D_X\|   \arrow{dl}{\htpyeq}  \arrow[ll, "\htpyeq", yshift = 3, swap] \\
\, &   \|D_A \star_{\tilde{R}} D_X \|  & \, 
\end{tikzcd} 
\]

\subsection{Proofs of Nerve Lemma via relational complexes}
\label{section:nerve_lemma_pf_via_relational_complexes}
Here, we present proofs of the Nerve Lemma via relational joins and relational products. 

\begin{theorem}[Nerve Lemma via relational join] Let $K$ be a simplicial complex, and let $\mathcal{U}$ be a good cover of $K$ via subcomplexes. Let $P_K$ denote the face poset of $K$, and let $P_{N\mathcal{U}}$ denote the face poset of the nerve $N\mathcal{U}$. Let $\tilde{R} \subseteq P_K \times P_{N\mathcal{U}} $ be the covering relation. The inclusion maps
\[ P_K \xrightarrow{i} P_K \star_{\tilde{R}} P^{\op}_{N\mathcal{U}} \xleftarrow{j} P^{\op}_{N\mathcal{U}} \]
induce homotopy equivalences
\begin{equation}
\label{eq:nerve_relationaljoin}
\| \Delta P_K \| \xrightarrow{\htpyeq} \|\Delta (P_K \star_{\tilde{R}} P^{\op}_{N\mathcal{U}})\| \xleftarrow{\htpyeq} \|\Delta P^{\op}_{N \mathcal{U}}\|.
\end{equation}
\end{theorem}

\begin{proof}
Consider $i: P_K \to P_K \star_{\tilde{R}} P^{\op}_{N\mathcal{U}}$ and fibers of the form $i^{-1}_{\leq \sigma}$.
If $\sigma_K \in P_K$, considered as an element of $P_K \star_{\tilde{R}} P^{\op}_{N\mathcal{U}}$,
then $i^{-1}_{\leq \sigma_K}$ has $\sigma_K$ as a maximum element. So $i^{-1}_{\leq \sigma_K}$ is contractible.  If $\sigma_{\mathcal{U}} \in P^{\op}_{N\mathcal{U}}$, then $i^{-1}_{\leq \sigma_{\mathcal{U}}}$ consists of all $\sigma_K \in P_K$ that are covered by elements corresponding to $\sigma_{\mathcal{U}}$. Then, $\Delta i^{-1}_{\leq \sigma_{\mathcal{U}}}$ is a barycentric subdivision of the subcomplex covered by $\sigma_{\mathcal{U}}$, which is contractible by assumption.

Now, consider $j:P^{\op}_{N\mathcal{U}} \to P_K \star_{\tilde{R}} P^{\op}_{N\mathcal{U}}$ and fibers of the form $j^{-1}_{\geq \sigma}$. If $\sigma_{\mathcal{U}} \in P^{\op}_{N\mathcal{U}}$, then $j^{-1}_{\geq \sigma_{\mathcal{U}}}$ has a minimum element, namely $\sigma_{\mathcal{U}}$. 
If $\sigma_K \in P_K$, then 
$$j^{-1}_{\geq \sigma_K} = \{ \sigma_{\mathcal{U}} \in P^{\op}_{N\mathcal{U}} \, | \, \sigma_{\mathcal{U}} \text{ covers } \sigma_K \}.$$ 
Let $\mathcal{U}_{\sigma_K} \subseteq \mathcal{U}$ be the subset $\mathcal{U}_{\sigma_K} = \{ U \in \mathcal{U} \, | \, U \text{ covers } \sigma_K \}.$ Then, $j^{-1}_{\geq \sigma_K}$ is the collection of nonempty finite subsets of $\mathcal{U}_{\sigma_K}$ and is therefore contractible.

By Quillen's Poset Fiber Lemma (Lemma~\ref{lemma:Quillen_fiber}), $i$ and $j$ induce homotopy equivalences in the order complexes, establishing Equation~\ref{eq:nerve_relationaljoin}. Since $\Delta P_K$ is the barycentric subdivision of $K$ and $\Delta P_{N\mathcal{U}}$ is the barycentric subdivision of $N\mathcal{U}$, the Nerve Lemma follows. 
\end{proof}

\begin{theorem}[Nerve Lemma via relational products] Let $K$ be a simplicial complex, and let $\mathcal{U}$ be a good cover of $K$ via subcomplexes. Let $P_K$ denote the face poset of $K$, and let $P_{N\mathcal{U}}$ denote the face poset of the nerve $N\mathcal{U}$. Let $\tilde{R} \subseteq P_K \times P_{N\mathcal{U}} $ be the covering relation, considered as a poset.
The projection maps
\[ P_K \leftarrow \tilde{R} \rightarrow P_{N\mathcal{U}} \]
induce homotopy equivalences
\[ \|\Delta P_K \| \xleftarrow{\htpyeq}  \| \Delta \tilde{R}\| \xrightarrow{\htpyeq} \| \Delta P_{N \mathcal{U}} \|.\] 
\end{theorem}

\begin{proof}
Let $\sigma \in P_K$. Then $\tilde{R}_{\sigma} = \{\tau \in P_{N \mathcal{U}} \, | \sigma \tilde{R} \tau \}$ consists of nonempty finite subsets of $\mathcal{U}_{\sigma} =\{ U \in \mathcal{U} \, | \, U \text{ covers } \sigma \}$, and is therefore contractible.

Let $\tau \in P_{N\mathcal{U}}$. Then, $\tilde{R}_{\tau} = \{\sigma \in P_K \, | \, \sigma \tilde{R} \tau \}$ is contractible by assumption. The theorem follows from Quillen's Product Fiber Lemma (Lemma~\ref{lemma:QuillenProductFiber}).
\end{proof}
Note that this Quillen-style proof of the Nerve Lemma appears in \cite{bjornerHomotopyTypePosets1981} (Lemma 1.1).

\section{The relational complexes and long exact sequences}
\label{section:SpectralSequences}

So far, given a relation $\tilde{R} \subseteq P_K \times P_M$ between face posets of simplicial complexes,
we constructed the relational join complex $K \star_{\tilde{R}} M$ and the relational product complex $K \times_{\tilde{R}} M$. When the relations arise in the context of Dowker duality or good covers, the relational join complex and the relational product complex are all homotopy equivalent to $K$ and $M$. Here, we consider general relations $\tilde{R}$ and show that the relational join complex and the product complex are related via a double mapping cylinder and that the homologies of all these complexes fit together in a long exact sequence. 

\begin{theorem}
\label{thm:double_mapping_cylinder}
Let $\tilde{R} \subseteq P_K \times P_M$ be a relation between the face posets of simplicial complexes $K$ and $M$. Let $K \star_{\tilde{R}} M$ and $K \times_{\tilde{R}} M$ each denote the relational join complex and the relational product complex. Given the projection map $\|K \| \times \|M \| \to \|K\|$, let $p_K: K \times_{\tilde{R}} M \to \| K \|$ be its restriction to the subcomplex $K \times_{\tilde{R}} M$, and let $p_M:  K \times_{\tilde{R}} M \to \| M \|$ be defined similarly. Then, the relational join $ \| K \star_{\tilde{R}}M \|$ is homeomorphic to the double mapping cylinder of $\|K\| \xleftarrow{p_K} K \times_{\tilde{R}} M \xrightarrow{p_M} \|M\|$. 
\end{theorem} 

See Figure~\ref{fig:double_mapping_cylinder} for an illustration. Note that the join $\| K \star M \|$ is the double mapping cylinder of $\|K\| \leftarrow \| K\| \times \|M\| \to \|M\|$, where the maps are projection maps. We obtain the relational join complex by restricting the diagram to $\|K\| \xleftarrow{p_K} K \times_{\tilde{R}} M \xrightarrow{p_M} \|M\|$.  

\begin{figure}[h]
    \centering
    \includegraphics[width=0.35\linewidth]{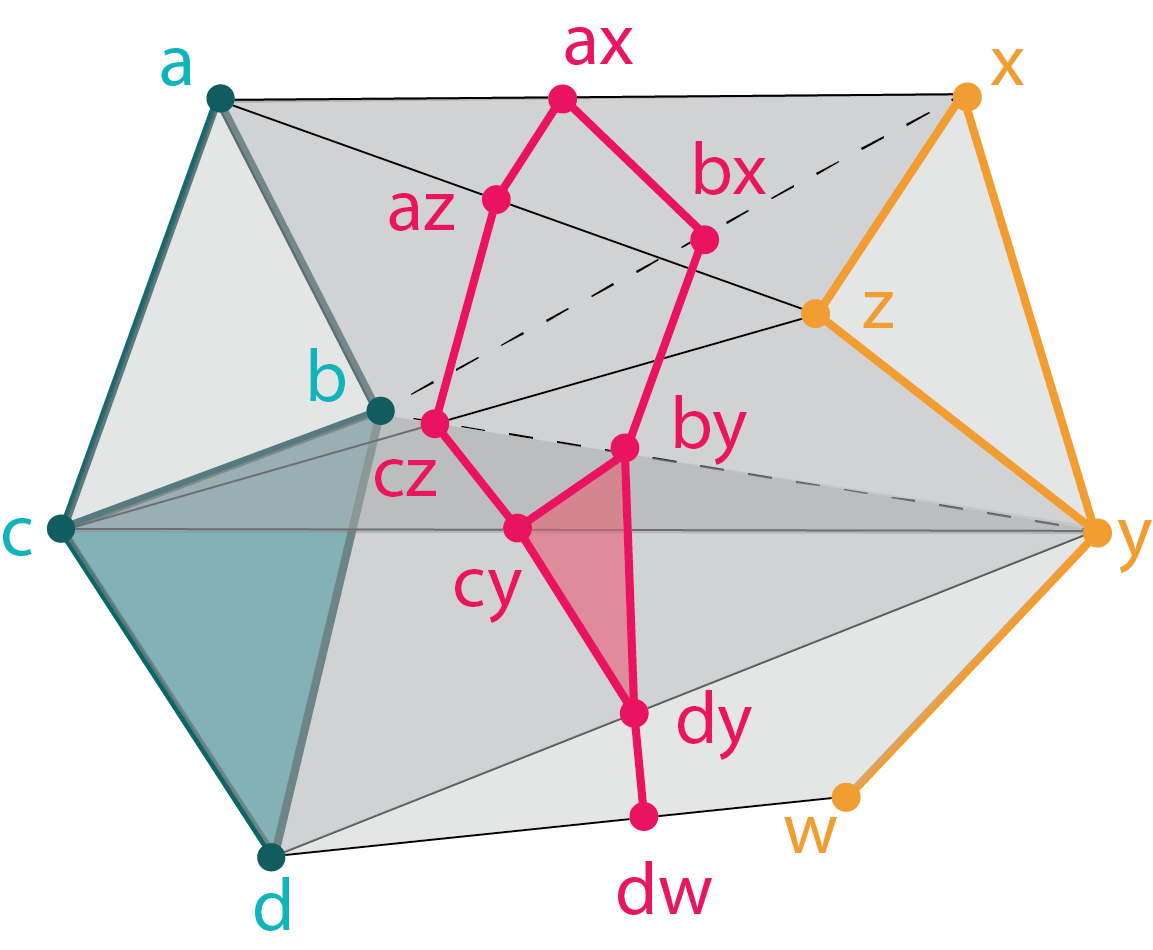}
    \caption{The relational complexes for Example~\ref{DowkerJoinExample2}. The relational join complex $\|K \star_{\tilde{R}} M\|$ contains the relational product complex $K \times_{\tilde{R}} M$, shown in red.  }
    \label{fig:double_mapping_cylinder}
\end{figure}

\begin{proof}
Note that the geometric realization $\| K \star_{\tilde{R}}M \|$ consists of $\|\sigma\|$, $\|\tau\|$, and $\|\sigma \cup \tau\|$ for simplices $\sigma \in K$ and $\tau \in M$ where $\sigma \tilde{R} \tau$. In particular, each $\| \sigma \cup \tau\|$ is the union of line segments between points of $\|\sigma\|$ and $\| \tau \|$.

Since $K \times_{\tilde{R}}M$ is a subcomplex of $\|K\| \times \|M\|$, we use $(k,m)$ to refer to points in $K \times_{\tilde{R}}M$. The double mapping cylinder of $\|K\| \xleftarrow{p_K} K \times_{\tilde{R}} M \xrightarrow{p_M} \|M \|$ 
is the quotient space 
\begin{equation}
\label{eq:quotient_space}
\Big((K \times_{\tilde{R}} M) \times I \Big)  \sqcup \|K \| \sqcup \| M\|  \, / \, \sim \,,
\end{equation}
where $I$ is the unit interval and the equivalence relation is generated by $(k, m, 0) \sim p_K(k,m) = k$ and $(k, m, 1) \sim p_M(k,m) = m$. Let $q$ be the quotient map  
\[ q: \Big( (K \times_{\tilde{R}} M) \times I \Big) \sqcup \|K \| \sqcup \| M\|  \to \Big( (K \times_{\tilde{R}} M) \times I \Big)  \sqcup \|K \| \sqcup \| M\|  \, / \, \sim \, .\]

We will show that the relational join $\| K \star_{\tilde{R}} M\|$ is homeomorphic to the above quotient space. Define 
\[g:\Big((K \times_{\tilde{R}} M) \times I \Big)  \sqcup \|K \| \sqcup \| M\| \to \| K \star_{\tilde{R}} M\|\]
by 
\[
 g|_{ (K \times_{\tilde{R}} M) \times I}(k, m, t) = (1-t)\,k + t\, m,
  \qquad
  g|_{\|K\|} = \iota_K,
  \qquad
  g|_{\|M\|} = \iota_M,
\]
where $\iota_K$ and $\iota_M$ are inclusion maps of $\|K\|$ and $\|M\|$ into $\|K \star_{\tilde{R}} M \|$.
We will show that $g$ is also a quotient map that makes the same identifications as $q$.

The maps $g|_{\|K\|}$ and $g|_{\|M\|}$ are continuous. For each closed cell $\|\sigma\| \times \|\tau\|$ of $(K \times_{\tilde{R}} M)$, $g$ continuously maps $\|\sigma\| \times \|\tau\| \times I$ onto $\| \sigma \cup \tau\|$ 
and includes into $\| K \star_{\tilde{R}}M\|$. Thus, $g|_{(\|\sigma\| \times \|\tau\|) \times I}$ is continuous, so $g|_{(K \times_{\tilde{R}} M) \times I}$ is continuous (\cite{lundellTopologyCWComplexes1969} Ch. II Corollary 5.4). So $g$ is continuous and surjective. 

We now show that $g$ takes saturated closed sets to closed sets. Recall that $K \times_{\tilde{R}}M$ and $\| K \star_{\tilde{R}} M \|$ both have the weak topology with respect to their respective closed cells. Let $C \subseteq \|K \star_{\tilde{R}}M \|$ so that $g^{-1}(C)$ is closed. So $g^{-1}(C) \cap \Big( (K \times_{\tilde{R}} M) \times I \Big)$ is closed in $(K \times_{\tilde{R}} M) \times I$, and its intersection with $\|\sigma \|\times \|\tau\| \times I$ is closed 
and therefore compact. 
$g$ maps $g^{-1}(C) \cap \Big((K \times_{\tilde{R}} M) \times I  \Big) \cap (\|\sigma\| \times \|\tau\| \times I)$ to $C  \cap \|\sigma \cup \tau \|$, so $C  \cap \| \sigma \cup \tau \|$ is compact. Since $C \cap \| \sigma \cup \tau\|$ is a compact subset of $\| K \star_{\tilde{R}} M \|$, which is Hausdorff,  $C \cap \| \sigma \cup \tau\|$ is closed. A similar argument shows that $C \cap \| \sigma \|$ and $C \cap \|\tau\|$ are closed for every simplex $\sigma$ of $K$ and $\tau$ of $M$. So $C$ is closed in $\| K \star_{\tilde{R}} M\|$. 

Since $g$ is a continuous, surjective map that takes saturated closed sets to closed sets, $g$ is a quotient map. 

Note that $q$ and $g$ make the same identifications. 
 By the uniqueness of quotient spaces, there exists a homeomorphism 
 \[ \Big((K \times_{\tilde{R}} M) \times I \Big) \sqcup \|K \| \sqcup \| M\|  \, / \, \sim \,\to \| K \star_{\tilde{R}} M\|.\]
\end{proof}

Applying the Mayer-Vietoris sequence to the double mapping cylinder (\cite{Hatcher2002} Example 2.48), we obtain the following long exact sequence. All homologies are computed with coefficients in some fixed abelian group $G$. 

\begin{corollary}
\label{cor:LES}
Let $\tilde{R} \subseteq P_K \times P_M$ be a relation between the face posets of simplicial complexes. 
Let $K \star_{\tilde{R}} M$ and $K \times_{\tilde{R}} M$ each denote the relational join complex and the relational product complex. There exists a long exact sequence
 \[  \cdots \to H_n(K \times_{\tilde{R}} M) \to H_n(K) \oplus H_n(M) \to H_n(K\star_{\tilde{R}}M) \to H_{n-1}(K \times_{\tilde{R}} M) \to \cdots \quad .\]
\end{corollary}

\subsection{Applications: covers of simplicial complexes}
We can now understand the relations between a simplicial complex and the nerve of its covering, even when the cover isn't a good cover. Let $K$ be a simplicial complex, and let $\mathcal{U}$ be any covering of $K$ with subcomplexes. Let $N\mathcal{U}$ be the nerve of the cover, 
and let $\tilde{R} \subseteq P_K \times P_{N\mathcal{U}}$ be the covering relation.

Since $\mathcal{U}$ isn't necessarily a good cover, $K$ is not necessarily homotopy equivalent to $N\mathcal{U}$. However, the homologies of $K$ and $N\mathcal{U}$ fit into a long exact sequence according to Corollary~\ref{cor:LES}
\[ \cdots \to  H_n(K \times_{\tilde{R}}N\mathcal{U}) \to H_n(K) \oplus H_n(N\mathcal{U}) \to H_n(K\star_{\tilde{R}} N\mathcal{U}) \to H_{n-1}(K \times_{\tilde{R}} N\mathcal{U}) \to \cdots \; . \]

\begin{example}
\normalfont
Let $K$ be the simplicial complex in Figure~\ref{fig:nongood_nerve}A. Let $\mathcal{U} =\{ x, y\}$ be a covering of $K$ (Figure~\ref{fig:nongood_nerve}A). 
Let $\tilde{R} \subseteq P_K \times P_{N\mathcal{U}}$ be the covering relation.
The relational join complex and the product complex are illustrated in Figure~\ref{fig:nongood_nerve}B. One can check that their homologies fit into the above long exact sequence. 
\begin{figure}[h!]
    \centering
    \includegraphics[width = 0.8\linewidth]{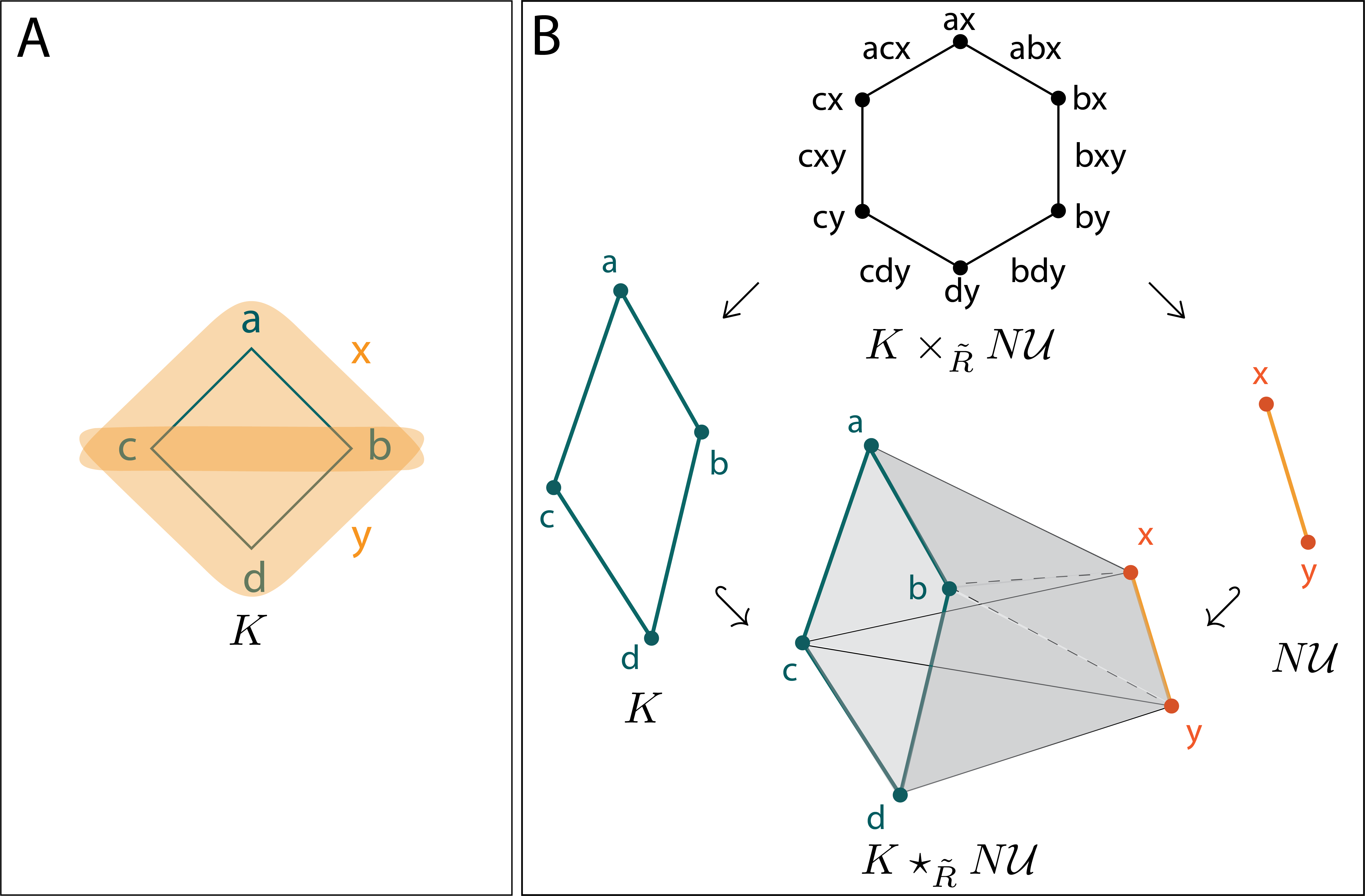}
    \caption{Example cover of a simplicial complex. \textbf{A.} A simplicial complex $K$ and a covering $\mathcal{U}$ that is not a good cover. \textbf{B.} Simplicial complex $K$ (left), the nerve $N\mathcal{U}$ (right), the relational product complex $K \times_{\tilde{R}} N\mathcal{U}$ (top), and the relational join complex $K \star_{\tilde{R}} N\mathcal{U}$ (bottom).}
    \label{fig:nongood_nerve}
\end{figure}

\end{example}

\noindent\textbf{Acknowledgments}
The author thanks Vin de Silva for constructive discussions.

\newpage
\printbibliography

\end{document}